\documentclass[11pt]{article}
\usepackage{amsmath}
\usepackage{amsthm}
\usepackage{amsfonts}
\usepackage{amssymb}
\usepackage{latexsym}
\usepackage{diagbox}
\usepackage{mathtools}
\usepackage{comment}

\usepackage{array}
\usepackage{subcaption}
\usepackage{booktabs}
\usepackage{tkz-graph}
\usepackage{graphics,graphicx}
\usepackage{multirow}
\usepackage{enumitem}
\usepackage{tikz}

\usepackage{tikz}
\usepackage{tikz-cd}

\usetikzlibrary{matrix,arrows}
\usetikzlibrary{fit}
\usetikzlibrary{patterns}
\usetikzlibrary{chains,fit,shapes}
\usetikzlibrary{positioning,calc}
\usetikzlibrary{graphs}
\usetikzlibrary{arrows,decorations.markings}
\usepackage[colorlinks,
linkcolor=blue,
anchorcolor=blue,
citecolor=blue
]{hyperref}

\newtheorem{thm}{Theorem}[section]   
\newtheorem{prop}[thm]{Proposition}
\newtheorem{cor}[thm]{Corollary}

\newtheorem{rem}[thm]{Remark}

\DeclareMathOperator{\asc}{asc}
\DeclareMathOperator{\des}{des}
\DeclareMathOperator{\rmax}{rmax}
\DeclareMathOperator{\lmax}{lmax}
\DeclareMathOperator{\rmin}{rmin}
\DeclareMathOperator{\lmin}{lmin}

\begin{document}

\begin{center}
{\large \bf  Counting permutations avoiding two flat partially ordered patterns}
\end{center}
\begin{center}
Shiqi Cao$^{1}$, Huihua Gao$^{2}$, Sergey Kitaev$^{3}$, and Yitian Li$^{4}$\\[6pt]

$^{1,2,4}$Center for Combinatorics, LPMC, Nankai University, Tianjin 300071, P. R. China

 $^{3}$Department of Mathematics, \\
Department of Mathematics and Statistics, University of Strathclyde, 26 Richmond Street, Glasgow G1 1XH, United Kingdom \\[6pt]

 Emails: $^{1}${\tt shiqicao@mail.nankai.edu.cn},
        $^{2}${\tt huihuagao@mail.nankai.edu.cn},
        $^{3}${\tt sergey.kitaev@strath.ac.uk},
       $^{4}${\tt yitianli@mail.nankai.edu.cn}.
\end{center}

\noindent\textbf{Abstract.}
Partially ordered patterns (POPs) play an important role in the study of permutation patterns, providing a convenient framework for describing large families of classical patterns. The problem of enumerating permutations that avoid POPs has therefore attracted considerable attention in the literature. In particular, Gao and Kitaev resolved many counting problems for POP-avoiding permutations of lengths 4 and 5, linking the enumeration to a wide range of other combinatorial objects.

Motivated by their work, we initiate the study of permutations that simultaneously avoid two POPs belonging to the class of flat POPs. We establish a connection between permutations avoiding such POPs and the $k$-Fibonacci numbers. Moreover, we provide a bijection between permutations avoiding these POPs and certain restricted permutations, which allows us to use the method developed by Balti\'{c} to derive the generating function for permutations avoiding these POPs. Finally, we obtain enumerative results for separable permutations avoiding these two POPs, of lengths up to 5, with respect to six statistics, thereby extending the results of Gao et al.\ on the avoidance of a single flat POP in separable permutations.  Notably, when both patterns are of length 5, the respective generating function is a rational function, with the sum in the numerator (resp., denominator) containing 293 (resp., 17) monomials.\\
	
\noindent {\bf Keywords:}  partially ordered pattern,  joint distribution,  $k$-Fibonacci number,  restricted permutation,  separable permutation
	
\section{Introduction}\label{intro-sec}
Let $S_n$ be the set of all permutations of $[n]:=\{1,2,\ldots,n\}$. A {\em (classical) pattern} is a permutation. A permutation $\pi_1\pi_2\cdots\pi_n\in S_n$ avoids a pattern $p=p_1p_2\cdots p_k\in S_k$  if there is no subsequence $\pi_{i_1}\cdots\pi_{i_k}$, where $1\le i_1<\cdots< i_k\le n$, such that $\pi_{i_j}<\pi_{i_r}$ if and only if $p_j<p_r$.
For example, the permutation $123465$ avoids the pattern $321$.
In this paper we use $\varepsilon$ to denote the empty permutation. For a permutation $\pi=\pi_1\pi_2\cdots\pi_n$, its {\em reverse} is the permutation $r(\pi)=\pi_n\pi_{n-1}\cdots \pi_1$ and its {\em complement} is the permutation $c(\pi)=c(\pi_1)c(\pi_2)\cdots c(\pi_n)$ where $c(x)=n+1-x$. Also, the length of a permutation $\pi$, denoted by $|\pi|$, is the number of elements in $\pi$. For example, for $\pi=423165$, $|\pi|=6$, $r(\pi)=561324$, and $c(\pi)=354612$. Throughout this paper, we often use \emph{g.f.} as an abbreviation for ``generating function''.
	
Define a \emph{partially ordered pattern} ({POP}) $P$ of length $k$ by a $k$-element partially ordered set (poset) labeled by the elements in $\{1,2,\ldots,k\}$. An occurrence of $P$ in a permutation $\pi_1\cdots\pi_n\in S_n$ is a subsequence $\pi_{i_1}\cdots\pi_{i_k}$, where $1\le i_1<\cdots< i_k\le n$,  such that $\pi_{i_j}<\pi_{i_m}$ if and only if $j<m$ in $P$. Thus, a classical pattern of length $k$ corresponds to a $k$-element chain. For example, the POP $p=$ \hspace{-3.5mm}
\begin{minipage}[c]{3.5em}
    \centering
    \begin{tikzpicture}[scale=0.5]
        \draw [line width=1](0,-0.5)--(0,0.5);
        
        \foreach \x/\y in {0/-0.5,1/-0.5,0/0.5}
            \draw (\x,\y) node [scale=0.4, circle, draw, fill=black]{};
        
        \node [left] at (0,-0.6){\footnotesize 3};
        \node [right] at (1,-0.6){\footnotesize 2};
        \node [left] at (0,0.6){\footnotesize 1};
    \end{tikzpicture}
\end{minipage}
	occurs six times in the permutation \(34152\), namely, as the subsequences \(341, 342, 312, 352, 412,\) and \(452\).
Clearly, avoiding the POP $P$ is the same as avoiding the patterns \(312\), \(321\), and \(231\) at the same time. We let $S_n(P_1,\ldots,P_m)$ denote the set of permutations in $S_n$ that simultaneously avoid the given patterns $P_1,\ldots,P_m$. POP-avoiding permutations have been studied by numerous authors; see, for example, \cite{BursteinKitaev2008,GaoKitaev2019,Gao2026,HKZ,Kitaev2007,WangYan,YapWehlauZaguia2021}.

In this paper, we let $P_j$ and $\widetilde{P}_\ell$ be the POPs of the form shown in Figure~\ref{pic-Pk}, known as {\em flat POPs}. We assume $j\geq 2$ and $\ell\geq 2$.

\begin{figure}[htbp]
  \begin{center}
\begin{tabular}{cc}  
\begin{tikzpicture}[scale=0.85]
    \coordinate (T) at (2.1,1.6);
    \coordinate (a) at (0,0);
    \coordinate (b) at (1.2,0);
    \coordinate (c) at (2.4,0);
    \coordinate (d) at (4.2,0);

    \draw[line width=1] (T)--(a);
    \draw[line width=1] (T)--(b);
    \draw[line width=1] (T)--(c);
    \draw[line width=1] (T)--(d);

    \node[circle,draw,fill=black,inner sep=1.2pt] at (T) {};
    \node[circle,draw,fill=black,inner sep=1.2pt] at (a) {};
    \node[circle,draw,fill=black,inner sep=1.2pt] at (b) {};
    \node[circle,draw,fill=black,inner sep=1.2pt] at (c) {};
    \node[circle,draw,fill=black,inner sep=1.2pt] at (d) {};

    \node[scale=0.15,circle,draw,fill=black] at (3.05,0) {};
    \node[scale=0.15,circle,draw,fill=black] at (3.30,0) {};
    \node[scale=0.15,circle,draw,fill=black] at (3.55,0) {};

    \node[above] at (T) {\small $1$};
    \node[below] at (a) {\small $2$};
    \node[below] at (b) {\small $3$};
    \node[below] at (c) {\small $4$};
    \node[below] at (d) {\small $j$};
  \end{tikzpicture}
  &
    \begin{tikzpicture}[scale=0.85]
    \coordinate (B)   at (2.6,0);
    \coordinate (u1)  at (0.8,1.6);
    \coordinate (u2)  at (2.0,1.6);
    \coordinate (u3)  at (3.2,1.6);
    \coordinate (uL)  at (5.0,1.6);

    \draw[line width=1] (B)--(u1);
    \draw[line width=1] (B)--(u2);
    \draw[line width=1] (B)--(u3);
    \draw[line width=1] (B)--(uL);

    \node[scale=0.4,circle,draw,fill=black] at (B)  {};
    \node[scale=0.4,circle,draw,fill=black] at (u1) {};
    \node[scale=0.4,circle,draw,fill=black] at (u2) {};
    \node[scale=0.4,circle,draw,fill=black] at (u3) {};
    \node[scale=0.4,circle,draw,fill=black] at (uL) {};

    \node[scale=0.15,circle,draw,fill=black] at (3.8,1.6) {};
    \node[scale=0.15,circle,draw,fill=black] at (4,1.6) {};
    \node[scale=0.15,circle,draw,fill=black] at (4.2,1.6) {};

    \node[below] at (B)  {$\ell$};
    \node[above] at (u1) {$1$};
    \node[above] at (u2) {$2$};
    \node[above] at (u3) {$3$};
    \node[above] at (uL) {$\ell-1$};
  \end{tikzpicture}
  \end{tabular}
  \end{center}
  \vspace{-0.5cm}
  \caption{The POPs $P_j$ (to the left) and $\widetilde{P}_\ell$ (to the right).} 
  \label{pic-Pk}
\end{figure}
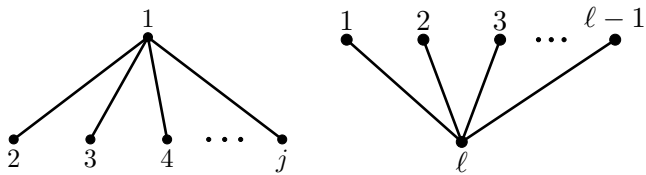

Note that $|S_n(P_j,\widetilde{P}_\ell)|=|S_n(P_\ell,\widetilde{P}_j)|$ for any $n \ge 0$, which can be demonstrated bijectively using the composition of reverse and complement operations. 

Recall that the Fibonacci numbers $\{F_n\}_{n\ge 0}$ are defined by $F_0=0$, $F_1=1$, and for $n\geq 2$,
\[
F_n=F_{n-1}+F_{n-2}.
\]
More generally, for $k\geq 2$, the $k$-Fibonacci numbers $\{F_n^{(k)}\}_{n\ge 0}$ are given by $F_1^{(k)}=1$, and for $n\geq 2$, adopting  the convention that $F_r^{(k)}=0$ for $r\leq 0$,
\begin{equation}\label{k-Fib-rec}
F_n^{(k)}=F_{n-1}^{(k)}+F_{n-2}^{(k)}+\cdots+F_{n-k}^{(k)}.
\end{equation}
In Section~\ref{s1}, we establish a connection between the $k$-Fibonacci numbers and permutations simultaneously avoiding $P_j$ and $\widetilde{P}_\ell$.

Let $\pi=\pi_1\pi_2\cdots\pi_n \in S_n$ and let $a,b$ be two integers. We call $\pi$ a \emph{restricted permutation} if, for all $i$ with $1 \le i \le n$, it satisfies $-a < \pi_i - i < b$.  
Denote the number of such permutations by $N(n,a,b)$.

In 2009, Kl{\o}ve~\cite{Klove2009} gave a method for deriving the generating function of $N(n,a,b)$ in the special case $a=b$. In 2010, Balti{\'c}~\cite{Bal2010} presented a method for obtaining the generating function of $N(n,a,b)$ for $a \le b$ by solving a system of equations.  
We establish a connection between restricted permutations and permutations simultaneously avoiding $P_j$ and $\widetilde{P}_{\ell}$ in Section~\ref{s3}.

Let $\pi=\pi_1\pi_2\cdots\pi_m \in S_m$ and $\sigma=\sigma_1\sigma_2\cdots\sigma_n \in S_n$. 
The \emph{direct sum}\index{direct sum $\oplus$} $\oplus$ and the \emph{skew sum}\index{skew sum $\ominus$} $\ominus$ are defined by constructing the permutations $\pi\oplus\sigma$ and $\pi\ominus\sigma$ as follows:
\begin{eqnarray*}
(\pi\oplus\sigma)_i &=& \left\{
\begin{array}{ll}
\pi_i, & \mbox{if } 1 \le i \le m,\\
\sigma_{i-m}+m, & \mbox{if } m+1 \le i \le m+n,
\end{array}
\right.\\
(\pi\ominus\sigma)_i &=& \left\{
\begin{array}{ll}
\pi_i+n, & \mbox{if } 1 \le i \le m,\\
\sigma_{i-m}, & \mbox{if } m+1 \le i \le m+n.
\end{array}
\right.
\end{eqnarray*}
For instance, $231\oplus312=231645$ and $231\ominus312=564312$.

The \emph{separable permutations} are those that can be built from the permutation $1$ by repeatedly applying the operations $\oplus$ and $\ominus$.
We assume that the empty permutation $\varepsilon$ is separable, although some papers make the opposite assumption.
Bose, Buss, and Lubiw~\cite{BBL98} introduced the notion of separable permutations in 1998, and it is well known~\cite{Kitaev2011Patterns} that the set of all separable permutations of length $n \ge 1$ is precisely $S_n(2413,3142)$. Note that the set of separable permutations is closed under both the reverse and complement operations.

Separable permutations appear in the literature in various contexts (see, for example, \cite{AAV2011,AHP2015,AJ2016,FuLinZeng2018,GaoLiu2024}).
It is known, and easy to see, that every permutation $\pi \in S_n(2413,3142)$ admits a decomposition that can be represented schematically as in Figure~\ref{sepStructure}, where $\pi$ is represented by its permutation diagram with a dot at $(i,\pi_i)$ for each $i$. Thus,
\begin{align*}
\pi = L_1 L_2 \cdots L_m\, n\, R_m R_{m-1} \cdots R_1.
\end{align*}
Here,
\begin{itemize}
\item[(i)] For $1 \le i \le m$, both $L_i$ and $R_i$ are required to be nonempty, except for $L_m$ and $R_1$;
\item[(ii)] Each $L_i$ and each $R_i$ consists of consecutive values;
\item[(iii)] $R_1 < L_1 < R_2 < L_2 < \cdots < R_m < L_m$, where for two permutation blocks $A$ and $B$, the relation $A < B$ means that every element of $A$ is less than every element of $B$. In particular, $R_1$, if it is nonempty, contains $1$.
\end{itemize}

\noindent
For example, if $\pi=32176854$, then $L_1=321$, $L_2=76$, $R_1=\varepsilon$, and $R_2=54$.
	
	\begin{figure}[htbp]
		\begin{center}
			\begin{tikzpicture}[line width=0.5pt,scale=0.24]
				\coordinate (O) at (0,0);
				
				\path (25,1)  node {$n$};
				\draw [dashed] (O)--++(50,0);
				\fill[black!100] (O)++(25,0) circle(1.5ex);
				
				\draw (20,-1) rectangle (24,-3);
				\path (22,-2)  node {$L_m$};
				\path (13.5,-2)  node {possibly empty$\boldsymbol{\rightarrow}$};
				
				\draw [dashed] (0,-4)--++(50,0);
				\draw [dashed] (0,-8)--++(50,0);
				\draw [dashed] (0,-17)--++(50,0);
				\draw [dashed] (0,-21)--++(50,0);
				\draw [dashed] (0,-25)--++(50,0);
				
				\draw [dashed] (25,0)--++(0,-29);
				\draw [dashed] (31,0)--++(0,-29);
				\draw [dashed] (37,0)--++(0,-29);
				\draw [dashed] (43,0)--++(0,-29);
				
				\draw [dashed] (19,0)--++(0,-29);
				\draw [dashed] (13,0)--++(0,-29);
				\draw [dashed] (7,0)--++(0,-29);
				
				\draw (26,-5) rectangle (30,-7);
				\path (28,-6)  node {$R_m$};
				
				\fill[black!100] (O)++(17,-9) circle(1.0ex);
				\fill[black!100] (O)++(16,-10) circle(1.0ex);
				\fill[black!100] (O)++(15,-11) circle(1.0ex);
				\draw (8,-14) rectangle (12,-16);
				\path (10,-15)  node {$L_2$};
				
				\draw (2,-22) rectangle (6,-24);
				\path (4,-23)  node {$L_1$};

				\fill[black!100] (O)++(32,-11) circle(1.0ex);
				\fill[black!100] (O)++(33,-12) circle(1.0ex);
				\fill[black!100] (O)++(34,-13) circle(1.0ex);
				\draw (38,-18) rectangle (42,-20);
				\path (40,-19)  node {$R_2$};
				
				\draw (45,-26) rectangle (49,-28);
				\path (47,-27)  node {$R_1$};
				\path (38.5,-27)  node {possibly empty$\boldsymbol{\rightarrow}$};
				
			\end{tikzpicture}
			\caption{Stankova's decomposition of separable permutation~\cite{Stankova1994}. Each $L_i$ and $R_j$ is a separable permutation.}\label{sepStructure}
		\end{center}
	\end{figure}
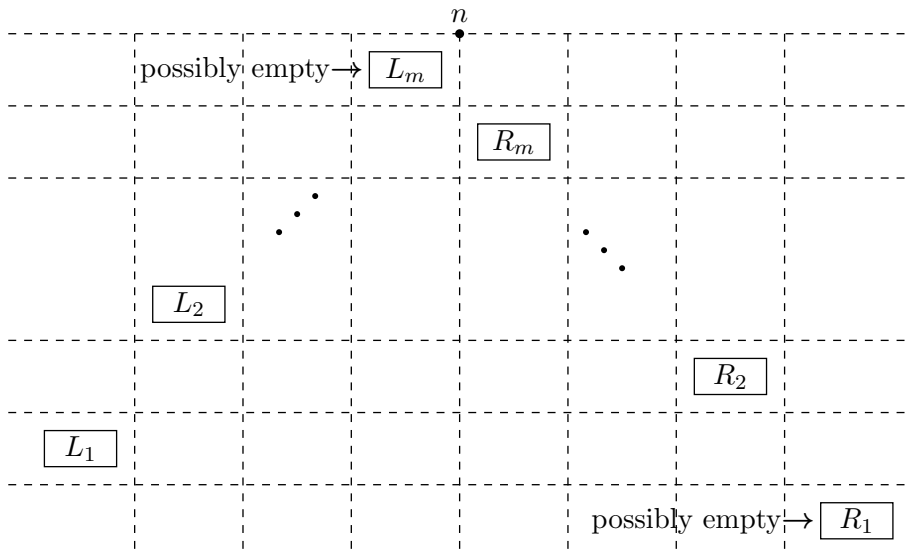

Let $\pi=\pi_1\pi_2\cdots\pi_n \in S_n$. For each $1 \le i \le n-1$, we call $i$ an \emph{ascent} of $\pi$ if $\pi_i < \pi_{i+1}$, and a \emph{descent} if $\pi_i > \pi_{i+1}$. The numbers of ascents and descents of $\pi$ are denoted by $\asc(\pi)$ and $\des(\pi)$, respectively. An element $\pi_i$ is a \emph{right-to-left maximum} (resp., \emph{right-to-left minimum}) if it is larger (resp., smaller) than every element to its right. We write $\rmax(\pi)$ and $\rmin(\pi)$, respectively, for the numbers of right-to-left maxima and minima in $\pi$. Dually, an element $\pi_i$ is a \emph{left-to-right maximum} (resp., \emph{left-to-right minimum}) if it is larger (resp., smaller) than every element to its left. The corresponding counts are denoted by $\lmax(\pi)$ and $\lmin(\pi)$. For example, if $\pi=31452$, then 
\[
\mathrm{lrmax}(\pi)=3, \quad 
\mathrm{lrmin}(\pi)=\mathrm{rlmin}(\pi)=\mathrm{rlmax}(\pi)=\mathrm{asc}(\pi)=\mathrm{des}(\pi)=2.
\]

Gao, Kitaev, Li, and Ruan~\cite{Gao2026} studied separable permutations avoiding a single POP and derived generating functions for the simultaneous distribution of the statistics $\mathrm{lrmax}$, $\mathrm{lrmin}$, $\mathrm{rlmax}$, $\mathrm{rlmin}$, $\mathrm{des}$, and $\mathrm{asc}$. We extend the results in~\cite{Gao2026} by studying the g.f.\
\begin{eqnarray*}
F_{j,\ell}(x,p,q,u,v,s,t)
:= \sum_{n \ge 0} x^n \sum_{\sigma \in S_n(2413,3142,P_j,\widetilde{P}_{\ell})}
p^{\mathrm{asc}(\sigma)} q^{\mathrm{des}(\sigma)} u^{\mathrm{lmax}(\sigma)} v^{\mathrm{rmax}(\sigma)} s^{\mathrm{lmin}(\sigma)} t^{\mathrm{rmin}(\sigma)},
\end{eqnarray*}
which gives the simultaneous distributions of the six statistics on separable permutations that simultaneously avoid $P_j$ and $\widetilde{P}_{\ell}$. Note that under the composition of reverse and complement, the statistics $\mathrm{lmax}$ and $\mathrm{rmin}$ (resp., $\mathrm{lmin}$ and $\mathrm{rmax}$) are exchanged. Therefore,
\begin{eqnarray}\label{inv-com}
F_{j,\ell}(x,p,q,u,v,s,t)=F_{\ell,j}(x,p,q,t,s,v,u).    
\end{eqnarray}

Hence, in our derivations we may assume that $j \le \ell$.

If $j=2$ or $\ell=2$, only the increasing permutation $12\cdots n$ belongs to $ S_n(2413,3142,P_2,\widetilde{P}_{2})$, so for $j\geq 2$ and $\ell\geq 2$,
\begin{eqnarray}\label{eq:1.1}
    F_{2,\ell}(x, p, q, u, v, s, t)  =  F_{j,2}(x, p, q, u, v, s, t)  = 1 + \sum_{n\geq 1} p^{n-1} u^{n} v s t^{n} x^{n}
= 1 + \frac{uvstx}{1 - putx}.
\end{eqnarray}

A system of functional equations in Section~\ref{most-general-sec} determines $F_{j,\ell}(x,p,q,u,v,s,t)$ for all $3 \le j,\ell \le 5$, and we also compute explicit forms of these generating functions as follows. The cases $j=\ell=3$, $j=3$, $\ell=4$, and $j=3$, $\ell=5$ are given in Theorems~\ref{33-thm}, \ref{34-thm}, and \ref{35-thm}, respectively. The cases $j=\ell=4$ and $j=4$, $\ell=5$ are covered by Theorems~\ref{44-thm} and \ref{45-thm}, respectively. Finally, the case $j=\ell=5$ is given by  Theorem~\ref{55-thm}. Notably, our formula for $F_{4,4}(x,p,q,u,v,s,t)$ is a rational function, with the sum in the numerator (resp., denominator) containing 49 (resp., 7) monomials, our formula for $F_{4,5}(x,p,q,u,v,s,t)$ is a rational function, with the sum in the numerator (resp., denominator) containing 93 (resp., 10) monomials, and our formula for $F_{5,5}(x,p,q,u,v,s,t)$ is a rational function, with the sum in the numerator (resp., denominator) containing 293 (resp., 17) monomials.
    
\section{$k$-Fibonacci numbers and permutations avoiding $P_j$ and $\widetilde{P}_{\ell}$}\label{s1}
In this section, we introduce a relation between $k$-Fibonacci numbers and permutations simultaneously avoiding the patterns $P_j$ and $\widetilde{P}_{\ell}$. 

In the following theorem, recall that $|S_n(P_j,\widetilde{P}_3)| = |S_n(P_3,\widetilde{P}_j)|$ by the composition of the reverse and complement operations. Note also that in the case \(j=\ell=3\), avoiding the POPs is equivalent to simultaneously avoiding the classical patterns \(231\), \(312\), and \(321\). This case therefore follows from a known result, first established in \cite{SimSch}.

\begin{thm}\label{thmt3}
For $j \geq 3$ and $n\geq 0$, $a_n := |S_n(P_j,\widetilde{P}_3)|=|S_n(P_3,\widetilde{P}_j)|=F^{(j-1)}_{n+1}$.
\end{thm}

\begin{proof}
It is sufficient to consider $S_n(P_j,\widetilde{P}_3)$. Suppose $\pi=\pi_1\pi_2\cdots\pi_n \in S_n(P_j,\widetilde{P}_3)$. We must have $\pi_1=1$ or $\pi_2=1$; otherwise, $\pi_1\pi_2 1$ forms an occurrence of $\widetilde{P}_3$. 

\noindent \textbf{Case 1:} $\pi_1 = 1$. In this case, $\pi_1$ cannot be part of an occurrence of $P_j$ or $\widetilde{P}_3$ in $\pi$, so we only need to ensure that $\pi_2\pi_3\cdots\pi_n$ avoids $P_j$ and $\widetilde{P}_3$. There are $a_{n-1}$ such permutations.

\noindent \textbf{Case 2:} $\pi_2 = 1$. In this case, $\pi_1 \in \{2, 3, \dots, j-1\}$; otherwise, the elements in the set $\{1,2,\ldots,j-1,\pi_1\}$ would form an occurrence of $P_j$ in $\pi$, which is impossible. 

If $\pi_1=2$, then $\pi_1\pi_2 = 21$ cannot be involved in any occurrence of $P_j$ or $\widetilde{P}_3$ in $\pi$, and the remaining permutation $\pi_3\pi_4\cdots\pi_n$ can be chosen in $a_{n-2}$ ways.  

If $\pi_1=3$, we must have $\pi_3=2$, since otherwise the elements in the set $\{2,3,\pi_3\}$ would form an occurrence of $\widetilde{P}_{3}$, which is impossible. Then $\pi_1\pi_2\pi_3 = 312$ does not affect the rest of $\pi$ ($312$ cannot be involved in any occurrence of $P_j$ or $\widetilde{P}_3$), and there are $a_{n-3}$ ways to choose $\pi_4\pi_5\cdots\pi_n$.  

In general, if $\pi_1=s$ with $s \le j-1$, we have $\pi_1\pi_2\cdots\pi_s = s123\cdots(s-1)$, and there are $a_{n-s}$ ways to choose $\pi_{s+1}\pi_{s+2}\cdots\pi_n$.

The cases above give the recurrence relation
\[
a_n = a_{n-1} + a_{n-2} + a_{n-3} + \dots + a_{n-j+1}.
\]
The desired result now follows by comparing this recurrence with \eqref{k-Fib-rec} and observing that the initial conditions are satisfied (the empty permutation $\varepsilon$ avoids the patterns). 
\end{proof}

\section{Restricted permutations and permutations avoiding $P_j$ and $\widetilde{P}_{\ell}$}\label{s3}

The following theorem gives a relation between restricted permutations and permutations that simultaneously avoid $P_j$ and $\widetilde{P}_\ell$.

\begin{thm}\label{bij}
We have $\pi=\pi_1\cdots\pi_n\in S_n(P_j,\widetilde{P}_{\ell})$ if and only if $-\ell+2\le \pi_i-i\le j-2$.
\end{thm}

\begin{proof}  First, we prove that $\pi_i - i \le j-2$ if and only if $\pi$ avoids $P_j$. Assume that there exists $i_0 \in [n]$ such that $\pi_{i_0} = i_0 + j - 1$. Let
\[
A_1 = \{1,2,\dots,i_0 + j - 2\}
\quad \text{and} \quad
B_1 = \{\pi_{i_0+1}, \pi_{i_0+2}, \dots, \pi_n\}.
\]
To avoid the POP $P_j$, the number of elements in $B_1$ that are less than $\pi_{i_0}$ must not exceed $j-2$. That is, $|A_1 \cap B_1| \le j-2$. Let $C_1 = \{\pi_1, \pi_2, \dots, \pi_{i_0-1}\}$. Then
\[
|A_1 \cap C_1| \ge (i_0 + j - 2) - (j-2) = i_0,
\]
which leads to a contradiction.

Conversely, assume that there exists an $i_0 \in [n]$ that is involved in an occurrence of the POP $P_j$. Then, for $i>i_0$, there are at least $j-1$ elements that are less than $\pi_{i_0}$. Let
\[
A_2 = \{1,2,\dots,i_0+j-2\}, \quad
B_2 = \{\pi_{i_0}, \pi_{i_0+2}, \dots, \pi_n\},
\quad \text{and} \quad
C_2 = \{\pi_1, \pi_2, \dots, \pi_{i_0-1}\}.
\]
 we have $|A_2 \cap B_2| \ge j$. Consequently,
\[
|A_2 \cap C_2| \le (i_0 + j - 2) - j = i_0 - 2.
\]
Thus, there must be at least one element of $C_2$ that is greater than $i_0 + j - 2$, which leads to a contradiction.

Similarly, we prove that $-\ell+2 \le \pi_i - i$ if and only if $\pi$ avoids $\widetilde{P}_{\ell}$. Assume that there exists $k_0 \in [n]$ such that $\pi_{k_0} = k_0 - \ell + 1$. Let
\[
A'_1 = \{k_0 - \ell + 2, \dots, n\}, \quad
B'_1 = \{\pi_1, \dots, \pi_{k_0-1}\}, \quad
\text{and} \quad
C'_1 = \{\pi_{k_0+1}, \dots, \pi_n\}.
\]
To avoid the POP $\widetilde{P}_{\ell}$, the number of elements in $A'_1$ that are less than $\pi_{k_0}$ must not exceed $\ell-2$. That is, $|A'_1 \cap B'_1| \le \ell-2$. Consequently,
\[
|A'_1 \cap C'_1| \ge |A'_1| - (\ell-2) = n - (k_0 - \ell + 2) - (\ell - 2) = n - k_0,
\]
which is impossible. The converse is proved analogously to the first part and is therefore omitted. The theorem follows.
\end{proof}

Let $\hat{N}(n,\ell,j)$ be the number of restricted permutations in $S_n$ satisfying $-\ell+2 \le \pi_i - i \le j-2$. Balti\'{c}~\cite{Bal2010} presented a method for deriving the g.f.\ of $\hat{N}(n,\ell,j)$ in the case $\ell \le j$, which is also the g.f.\ of $\hat{N}(n,j,\ell)$ by Theorem~\ref{bij} and the fact that $|S_n(P_j,\widetilde{P}_\ell)| = |S_n(P_\ell,\widetilde{P}_j)|$. Thus, this method can be employed to compute $|S_n(P_j,\widetilde{P}_\ell)|$. Finally, we note that the connection established by Balti\'{c}~\cite{Bal2010} between restricted permutations and the $k$-Fibonacci numbers is consistent with Theorem~\ref{bij} and our findings in Section~\ref{s1}.

\section{The g.f.\ $F_{j,\ell}(x,p,q,u,v,s,t)$ for any $3 \leq j,\ell \leq 5$}\label{most-general-sec}

Since the POPs $P_3$ and $\widetilde{P}_3$ occur in the patterns $2413$ and $3142$ defining separable permutations, we have the following proposition.

\begin{prop}\label{Proposition1}
For any $j\ge3$ and $\ell\ge3$, we have
\begin{align*}
S_n(2413,3142, P_j, \widetilde{P}_{3}) &= S_n(P_j, \widetilde{P}_{3}),\\
S_n(2413,3142, P_3, \widetilde{P}_{\ell}) &= S_n(P_3, \widetilde{P}_{\ell}).
\end{align*}
\end{prop}

Proposition~\ref{Proposition1} allows us to use the method developed in~\cite{Bal2010} to compute the generating functions.  However, Balti\'{c}~\cite{Bal2010} does not consider any statistics. In this section, we derive a system of functional equations for $F_{j,\ell}(x,p,q,u,v,s,t)$ for any $3 \leq j,\ell \leq 5$. To do this, for any $3 \leq j,\ell \leq 5$ and $0\leq i \leq j-2$, we define
\begin{align*}
f_{j,\ell,i}(x,p, q, u,v,s,t) := &\sum_{n \geq 1} x^n \sum_{\pi} p^{\asc(\pi)}q^{\des(\pi)}u^{\lmax(\pi)}v^{\rmax(\pi)}s^{\lmin(\pi)}t^{\rmin(\pi)},\\
g_{j,\ell,i}(x,p, q, u,v,s,t) :=& \sum_{n \geq 2} x^n \sum_{\sigma} p^{\asc(\sigma)}q^{\des(\sigma)}u^{\lmax(\sigma)}v^{\rmax(\sigma)}s^{\lmin(\sigma)}t^{\rmin(\sigma)},
\end{align*}
where the sums are over all $\pi$ (resp., $\sigma$)  in $S_n(2413, 3142, P_j,\widetilde{P}_{\ell})$ with $i$ elements to the right of $n$
and  with the element $n-1$, if $n-1$ exists, to the left (resp., right) of $n$. We will use $f_{j,\ell,i}$ (resp., $g_{j,\ell,i}$) to represent $f_{j,\ell,i}(x, p, q, u, v, s, t)$ (resp., $g_{j,\ell,i}(x, p, q, u, v, s, t)$).

We also  define
\begin{align*}G_j(x,p,q,u,v,s,t):=\sum_{n\geq 0}x^n\sum_{\sigma \in S_n(2413,3142,P_j) }p^{\asc(\sigma)}q^{\des(\sigma)}u^{\lmax(\sigma)}v^{\rmax(\sigma)}s^{\lmin(\sigma)}t^{\rmin(\sigma)},
\end{align*}
\begin{align*}
\widetilde{G_{\ell}}(x,p,q,u,v,s,t):=\sum_{n\geq 0}x^n\sum_{\sigma \in S_n(2413,3142,\widetilde{P}_{\ell}) }p^{\asc(\sigma)}q^{\des(\sigma)}u^{\lmax(\sigma)}v^{\rmax(\sigma)}s^{\lmin(\sigma)}t^{\rmin(\sigma)}.
\end{align*}
We will use $G_j$ (resp., $\widetilde{G_{\ell}}$) to represent $G_j(x,p,q,u,v,s,t)$ (resp., $\widetilde{G_{\ell}}(x,p,q,u,v,s,t)$).


\begin{thm}\label{thm jll}
For any $3 \leq j,\ell \leq 5$,  the g.f.\ $F_{j,\ell}:=F_{j,\ell}(x, p,q,u, v,s,t)$ is given by the system
\begin{align}
F_{j,\ell} &= 1 + \sum\limits_{i=0}^{j-2} \Big(f_{j,\ell,i}+ g_{j,\ell,i}\Big), \label{eq:5.1}\\[6pt]
f_{j,\ell,i} &= \; pquvx\Bigg(\sum\limits_{a=1}^{\ell-3}U_a(p, q, u, 1, s, 1)x^aV_i(p, q, 1, v, s, t)x^i\Bigg)\nonumber \\
            &\quad + p^2quvx\Bigg(\sum\limits_{a=1}^{\ell-3}U_a(p, q, u, 1, 1, 1)x^aV_i(p, q, 1, v, 1, t)x^i\Bigg)\big(F_{j,\ell}(x,p,q,u,1,s,t)-1 \big),\label{eq:5.2}\\[6pt]
g_{j,\ell,i} &= quvsxW_i(p,q,1,v,s,t)x^i 
            + pquvxW_i(p,q,1,v,1,t)x^i\big(F_{j,\ell}(x,p,q,u,1,s,t)-1\big) \nonumber \\
            &\quad + \sum\limits_{h=1}^{\min\{i-1, \ell-4\}}pq^2uvxX_h(p,q,1,v,1,1)x^h\nonumber\\
            &\quad \cdot
               \sum\limits_{m=1}^{\ell-h-3}\big(Z_m(p,q,u,1,s,1)x^mY_{i-h}(p,q,1,v,s,t)x^{i-h}\big)\nonumber \\
           &\quad + \sum\limits_{h=1}^{\min\{i-1, \ell-4\}}p^2q^2uvxX_h(p,q,1,v,1,1)x^h\nonumber\\
            &\quad \cdot
               \sum\limits_{m=1}^{\ell-h-3}\big(Z_m(p,q,u,1,1,1)x^mY_{i-h}(p,q,1,v,1,t)x^{i-h}\big) \big(F_{j,\ell}(x,p,q,u,1,s,t)-1\big), \label{eq:5.3}
\end{align}
where $1\leq i \leq j-2$, $X_n:=X_n(p, q, u, v, s, t)$ is the coefficient of $x^n$ in the g.f.\ for separable permutations of length  $n$ with respect to six statistics: $\asc$, $\des$, $\lmax$, $\rmax$, $\lmin$, and $\rmin$, $U_n:=U_n(p, q, u, v, s, t)$ (resp., $V_n:=V_n(p, q, u, v, s, t)$, $W_n:=W_n(p, q, u, v, s, t)$, $Y_n:=Y_n(p, q, u, v, s, t)$, and $Z_n:=Z_n(p, q, u, v, s, t)$) is the coefficient of $x^n$ in 
$G_{j-i}$ (resp., $\widetilde G_{\ell-a-1}$, $\widetilde G_{\ell-1}$, $\widetilde G_{\ell-m-h-1}$, and $G_{j-(i-h)}$).
The initial conditions are
\begin{align}\label{eq:5.4}
\begin{cases}
f_{j,\ell,0}& = uvstx+puvtx\left(F_{j,\ell}(x, p, q, u, 1, s, t)-1\right),\\
g_{j,\ell,0}&= 0, \\
f_{2,2,0} &= \frac{uvstx}{1 - putx}.
\end{cases}
\end{align}
\end{thm}

\begin{proof}
We make use of Stankova’s decomposition of separable permutations~\cite{Stankova1994}, as illustrated in Figure~\ref{sepStructure}.

When $j=\ell=2$, the only permutation in $S_n(2413, 3142, P_2,\widetilde{P}_{2})$ is $\pi=12\cdots (n-1)n$.
Clearly, there exists no element to the right of $n$ in $\pi$, hence the only value for $i$ is 0. Hence,
\begin{align*}
g_{2,2,0} = 0,
~
\mathrm{and}
~
f_{2,2,0}= \sum_{n\geq 1}p^{n-1}u^nvst^nx^n= \frac{uvstx}{1 - putx}.
\end{align*}
Moreover, we have
\begin{align*}
F_{2,2} &= 1+f_{2,2,0} +g_{2,2,0}
=1+\frac{uvstx}{1 - putx},
\end{align*}
which is consistent with  \eqref{eq:5.1}--\eqref{eq:5.3} and the initial conditions \eqref{eq:5.4}.

When  $3 \leq j,\ell \leq 5$,
it is clear that
\begin{align*}
F_{j,\ell}
&=1+\sum_{i=0}^{j-2} f_{j,\ell,i}+\sum_{i=0}^{j-2} g_{j,\ell,i}.
\end{align*}

\noindent  If $1 \leq i \leq j-2$, let $\pi$ be a permutation in $S_n(2413, 3142, P_j,\widetilde{P}_{\ell})$ with $i$ elements to the right of $n$. Since $i\geq 1$, it is obvious that $n\geq 2$ and $n-1$ exist in $\pi$. We have two cases depending on whether the element $n-1$ is to the right or to the left of $n$ in $\pi$.

\noindent \textbf{Case 1.} $n-1$ is to the left  of $n$ in $\pi$, as shown in Figure~\ref{sepStructure_thm3_g_1}.
 Note that $A$ is non-empty since $n-1$ is in $A$. From $i\geq 1$, we get that $B$ is also non-empty.
There are two subcases to consider.
\begin{figure}[htbp]
     \centering
\begin{tikzpicture}[line width=0.5pt,scale=0.24]
	\coordinate (O) at (0,0);
		\path (15,1)  node {$n$};
		\draw [dashed] (-3,0)--++(30,0);
        \draw [dashed] (-3,-5)--++(30,0);
						\draw [dashed] (-3,-10)--++(30,0);
                            \draw [dashed] (-3,-15)--++(30,0);
						\fill[black!100] (O)++(15,0) circle(1.5ex);
						\draw (10,-1) rectangle (14,-4);
						\path (12,-2.5)  node {$A$};
						\draw (16,-6) rectangle (20,-9);
						\path (18,-7.5)  node {$B$};
						\draw (4,-11) rectangle (8,-14);
						\path (6,-12.5)  node {$C$};
                            \draw (22,-16) rectangle (26,-19);
						\path (24,-17.5)  node {$D$};
                            \draw (-2,-16) rectangle (2,-19);
                            \path (0,-17.5)  node {$E$};
						\draw [dashed] (15,0)--++(0,-20);
						\draw [dashed] (9,0)--++(0,-20);
						\draw [dashed] (3,0)--++(0,-20);
						\draw [dashed] (21,0)--++(0,-20);
					\end{tikzpicture}
					\caption{A schematic representation of the permutation diagrams for permutations in Case 1.}
        \label{sepStructure_thm3_g_1}
				\end{figure}
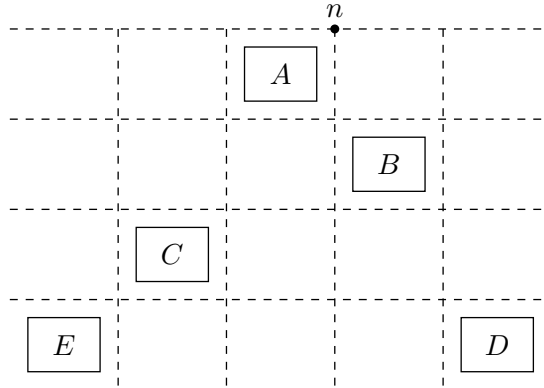
\begin{itemize}
\item[(a)]
$B$ contains $i$ elements,  which implies that $D$ is empty and hence $E$ is also empty. If $C$ is empty, then $A$  contains at most $\ell-3$ elements, or otherwise, $\ell-2$ elements in $A$ together with $n$ and an element in $B$ will form an occurrence of the pattern $\widetilde{P}_{\ell}$. Assume that $A$ contains $a$ elements, $1\leq a\leq \ell-3$, and recall that $U_n$ and $V_n$ are the coefficients of $x^n$ in $G_{j-i}$ and $\widetilde G_{\ell-a-1}$, respectively.
Since  $A$ does not contribute to the statistics rmax and rmin in $\pi$, and it can be
any non-empty separable permutation of length~$a$ avoiding  $P_{j-i}$, we have the g.f.\
$\sum\limits_{a=1}^{\ell-3} U_a(p, q, u, 1, s, 1)x^a$.
Moreover, $B$ does not contribute to the statistic lmax in $\pi$ and it can be
any non-empty separable permutation of length~$i$ avoiding  $\widetilde P_{\ell-a-1}$. The  g.f.\ is then $V_i(p,q,1,v,s,t)x^i$.
Note that $n$ contributes one extra ascent, one extra descent, one extra left-to-right maximum, and one extra right-to-left maximum.
Hence, in this subcase, we have the term
\begin{align}
&pquvx \sum_{a=1}^{\ell-3} U_a(p, q, u, 1, s, 1)x^a V_i(p,q,1,v,s,t)x^i.  \label{eq:5.5}
\end{align}

If $C$ is non-empty, then it is a separable permutation avoiding $P_j$ and $\widetilde P_{\ell}$ that 
does not contribute to the statistic rmax in $\pi$.
Thus, the g.f.\ is
$F_{j,\ell}(x,p,q,u,1,s,t)-1$.
Note that the subpermutation
$AnB$  does not contribute to the statistic
lmin  in $\pi$ and an extra ascent is formed between $AnB$ and $C$. Hence, from\eqref{eq:5.5}, in this subcase we have the term
\begin{align}\label{eq:5.6}
p^2quvx \sum_{a=1}^{\ell-3} U_a(p, q, u, 1, 1, 1)x^a V_i(p,q,1,v,1,t)x^i(F_{j,\ell}(x,p,q,u,1,s,t)-1).
\end{align}

\item[(b)]
$B$ contains $h$ elements,  where $1\leq h \leq i-1$ and $i \geq 2$. This implies that $C$ and $D$ are both non-empty, when  $3 \leq j,\ell \leq 5$, and they will result in an occurrence of the pattern $\widetilde P_{\ell}$, which is a contradiction. 
\end{itemize}

\noindent
Combining \eqref{eq:5.5} and \eqref{eq:5.6}, we have (\ref{eq:5.2}). \\[6pt] 
\noindent \textbf{Case 2.} $n-1$ is to the right  of $n$ in $\pi$, as shown in Figure~\ref{sepStructure_thm3_f_1}.
Note that $A$ is non-empty since $n-1$ belongs to $A$.
There are two subcases to consider.

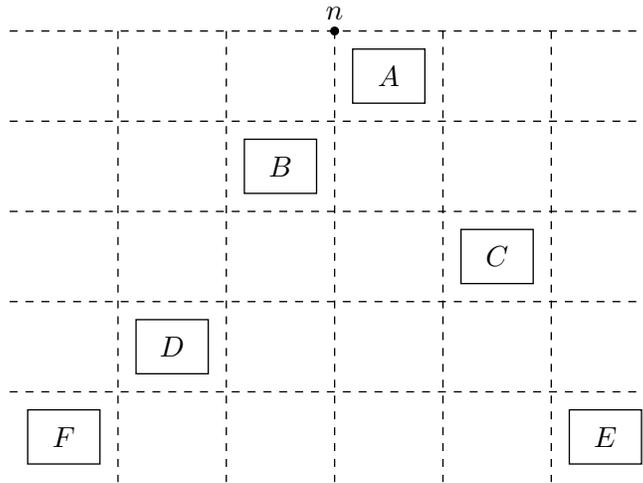
\begin{figure}[htbp]
\centering
 \begin{tikzpicture}[line width=0.5pt,scale=0.24]
		\coordinate (O) at (0,0);
		\path (15,1)  node {$n$};
		\draw [dashed] (-3,0)--++(36,0);
		\draw [dashed] (-3,-5)--++(36,0);
        \draw [dashed] (-3,-10)--++(36,0);
        \draw [dashed] (-3,-15)--++(36,0);
        \draw [dashed] (-3,-20)--++(36,0);
		\fill[black!100] (O)++(15,0) circle(1.5ex);

		\draw (16,-1) rectangle (20,-4);
		\path (18,-2.5)  node {$A$};
		\draw (10,-6) rectangle (14,-9);
	      \path (12,-7.5)  node {$B$};\
        \draw (22,-11) rectangle (26,-14);
	    \path (24,-12.5)  node {$C$};
        \draw (-2,-21) rectangle (2,-24);
        \draw (4,-16) rectangle (8,-19);
	    \path (6,-17.5)  node {$D$};
        \path (0,-22.5)  node {$F$};
        \draw (28,-21) rectangle (32,-24);
	    \path (30,-22.5)  node {$E$};
        \draw [dashed] (3,0)--++(0,-25);
		\draw [dashed] (15,0)--++(0,-25);
		\draw [dashed] (9,0)--++(0,-25);
		\draw [dashed] (21,0)--++(0,-25);
        \draw [dashed] (27,0)--++(0,-25);
		\end{tikzpicture}
        \caption{A schematic representation of the permutation diagrams for permutations in Case 2.}
        \label{sepStructure_thm3_f_1}
        \end{figure}

\begin{itemize}
\item[(a)]
$A$ contains $i$ elements,  which implies that $E\cup C$ and hence $D\cup F$ are both empty. If  $B$ is empty, the g.f. is
\vspace{-0.4cm}
\begin{align}\label{eq:5.8}
 quvsx  W_i(p,q,1,v,s,t)x^i,
\end{align} 
where recall that $W_n$ is the coefficient of $x^n$ in  $\widetilde G_{\ell-1}$. This is because
$A$ does not contribute to  the statistic lmax in $\pi$ and can be any separable permutation of length $i$ avoiding  $\widetilde P_{\ell-1}$,
while  $n$  contributes an extra descent, one extra left-to-right maximum, one extra right-to-left maximum, and one extra left-to-right minimum.

If  $B$ is non-empty, then $B$ can be any separable permutation simultaneously avoiding
$P_j$ and $\widetilde P_{\ell}$. $B$ does not contribute to rmax in $\pi$ and hence has the g.f.\ $F_{j,\ell}(x,p,q,u,1,s,t) - 1$.
Note that the subpermutation  $nA$ does not contribute to the statistic lmin in $\pi$
and an extra ascent is formed between $nA$ and $B$.
Hence, from \eqref{eq:5.8}, in this subcase we have 
\begin{align}\label{eq:5.9}
 p quvx W_i(p,q,1,v,1,t)x^i  \left(F_{j,\ell}(x,p,q,u,1,s,t) - 1\right).
\end{align}

\item[(b)]
$A$ contains $h$ elements,  where $1\leq h \leq i-1$ and $i \geq 2$. This implies that $B$ and $C$ are both  non-empty.
If $D$ is empty,  then $E$ and $F$ are empty.
Note that $A$ and $B$ contain  at most $\ell-3$ elements, or otherwise, $n$, elements in $A$ and $B$, and an element in $C$ will form an occurrence of $\widetilde{P}_{\ell}$,  which is a contradiction. Therefore, $B$ contains  at most $\ell-3-h$ elements. Since   $B$ is  non-empty, $h\leq\ell-4$. $A$ does not contribute to the statistic lmax, rmin, and lmin in $\pi$ and it can be any non-empty separable permutation of length~$h$. The  g.f. is then $\sum\limits_{h=1}^{\min\{i-1, \ell-4\}}X_h(p,q,1,v,1,1)x^h$.

Assume that $B$ contains $m$ elements, $1\leq m\leq \ell-3-h$, and recall that $Y_n$ (resp., $Z_n$) is the coefficient of $x^n$ in $\widetilde G_{\ell-m-h-1}$ (resp.,  $G_{j-(i-h)}$). Since  $B$ does not contribute to the statistics rmax and rmin in $\pi$, and it can be
any non-empty separable permutation of length~$m$ avoiding  $ P_{j-(i-h)}$, the g.f.\
is$\sum\limits_{m=1}^{\ell-h-3}Z_m(p, q, u, 1, s, 1)x^m$.
Since  $C$ does not contribute to the statistics rmax  in $\pi$, and it can be
any non-empty separable permutation of length~$i-h$ avoiding  $\widetilde P_{\ell-m-h-1}$, the g.f.\
is $Y_{i-h}(p,q,1,v,s,t)x^{i-h}$.
Moreover,
an extra descent is formed between  $A$
and $C$, while  $n$  contributes an extra ascent, an extra descent, one extra left-to-right maximum, and one extra right-to-left maximum. Hence, in this subcase we have the term
\begin{small}
\begin{align}\label{eq:5.10}
q \sum_{h=1}^{\min\{i-1, \ell-4\}}
  pquvx X_h(p,q,1,v,1,1)x^h  \sum\limits_{m=1}^{\ell-h-3}\left( Z_m(p, q, u, 1, s, 1)x^m
Y_{i-h}(p,q,1,v,s,t)x^{i-h} \right).
\end{align}
\end{small}

If $D$ is non-empty,
then $B$ and $C$ are both non-empty. For any $3 \leq j,\ell \leq 5$, $E$ must be empty, otherwise it will result in an occurrence of $\widetilde P_{\ell}$, so $F$ is empty. Then, the g.f.\ is
\begin{align}\label{eq:5.11}
&pq \sum_{h=1}^{\min\{i-1, \ell-4\}} pquvx X_h(p,q,1,v,1,1)x^h  
\sum\limits_{m=1}^{\ell-h-3} \bigl( Z_m(p, q, u, 1, 1, 1)x^m
 \nonumber \\
&\quad\cdot Y_{i-h}(p,q,1,v,1,t)x^{i-h} \bigr)
\bigl(F_{j,\ell}(x,p,q,u,1,s,t)-1\bigr)
\end{align}
\vspace{-0.3cm}
since\\[-5mm]
\begin{itemize}
\item[$\bullet$] $D$
is a non-empty separable
permutation avoiding $P_j$ and $\widetilde P_{\ell}$ and
not contributing to the statistic rmax in $\pi$.
Thus, the respective g.f.\ is
$F_{j,\ell}(x,p,q,u,1,s,t)-1$;
\item[$\bullet$] an extra ascent is formed between $B$ and $D$, and
\item[$\bullet$] because the subpermutation
$BnAC$  does not contribute to the statistic
lmin  in $\pi$, combining with \eqref{eq:5.10}, the respective g.f.\ is
\begin{align*}
&pq \sum_{h=1}^{\min\{i-1, \ell-4\}} pquvx X_h(p,q,1,v,1,1)x^h  
\sum\limits_{m=1}^{\ell-h-3} \bigl( Z_m(p, q, u, 1, 1, 1)x^m
 \nonumber \\
&\quad\cdot Y_{i-h}(p,q,1,v,1,t)x^{i-h} \bigr)
\bigl(F_{j,\ell}(x,p,q,u,1,s,t)-1\bigr).
\end{align*}

\end{itemize}
\end{itemize}

\noindent
Combining with \eqref{eq:5.8}--\eqref{eq:5.11}, we get (\ref{eq:5.3}).

If $i = 0$, let $\pi=\pi_1\cdots\pi_n$ be any non-empty separable permutation simultaneously avoiding $P_j$ and $\widetilde P_\ell$  with  $\pi_n=n$. Note that $n-1$, if  it exists, is to the left of $n$. We have $g_{j,\ell,0} = 0$ and
\begin{align}\label{eq:5.12}
f_{j,\ell,0} =uvstx+puvtx\left(F_{j,\ell}(x, p, q, u, 1, s, t)-1\right)
\end{align}
since if $n\geq 2$, the subpermutation $\pi_1 \pi_2\cdots\pi_{n-1}$  does not contribute to the statistic rmax and can be any permutation in $S_{n-1}(2413, 3142, P_j,\widetilde P_\ell)$, while $n$ contributes an ascent, an extra left-to-right maximum, an extra
right-to-left maximum, and an extra right-to-left minimum in~$\pi$.

Fixing integers  $3 \leq j,\ell \leq 5$ and  setting $v=1$ in \eqref{eq:5.1}--\eqref{eq:5.3}, we obtain, for $1\leq i \leq j-2$,
\vspace{-0.2cm}
\begin{align}
F_{j,\ell}(x,p,q,u,1,s,t) &= 1 + \sum\limits_{i=0}^{j-2} \Big(f_{j,\ell,i}(x,p,q,u,1,s,t) + g_{j,\ell,i}(x,p,q,u,1,s,t)\Big), 
\label{eq:5.13}\\[6pt]
f_{j,\ell,i}(x,p,q,u,1,s,t) &= pqux\Bigg(\sum\limits_{a=1}^{\ell-3}U_a(p,q,u,1,s,1)x^aV_i(p,q,1,1,s,t)x^i\Bigg) \nonumber \\
&\hspace{-3cm}\quad + p^2qux\Bigg(\sum\limits_{a=1}^{\ell-3}U_a(p,q,u,1,1,1)x^aV_i(p,q,1,1,1,t)x^i\Bigg)\big(F_{j,\ell}(x,p,q,u,1,s,t)-1\big), \label{eq:5.14}\\
g_{j,\ell,i}(x,p,q,u,1,s,t) &= qusxW_i(p,q,1,1,s,t)x^i 
            \nonumber \\
&\hspace{-3cm}\quad+ pquxW_i(p,q,1,1,1,t)x^i\big(F_{j,\ell}(x,p,q,u,1,s,t)-1\big) \nonumber \\
            & \hspace{-3cm} \quad + \sum\limits_{h=1}^{\min\{i-1, \ell-4\}}pq^2uxX_h(p,q,1,1,1,1)x^h
               \sum\limits_{m=1}^{\ell-h-3}\big(Z_m(p,q,u,1,s,1)x^mY_{i-h}(p,q,1,1,s,t)x^{i-h}\big)\nonumber \\
           &\hspace{-3cm} \quad + \sum\limits_{h=1}^{\min\{i-1, \ell-4\}}p^2q^2uxX_h(p,q,1,1,1,1)x^h
               \sum\limits_{m=1}^{\ell-h-3}\big(Z_m(p,q,u,1,1,1)x^mY_{i-h}(p,q,1,1,1,t)x^{i-h}\big)\nonumber\\
           &\hspace{-3cm} \quad \cdot \big(F_{j,\ell}(x,p,q,u,1,s,t)-1\big),
\label{eq:5.15}
\end{align}
with the initial conditions
\begin{align}\label{eq:5.16}
\begin{cases}
f_{j,\ell,0}(x, p, q, u, 1, s, t)& = ustx+putx\left(F_{j,\ell}(x, p, q, u, 1, s, t)-1\right),\\
g_{j,\ell,0}(x, p,q, u, 1,s,t)&= 0,\\
f_{2,2,0}(x, p, q, u, 1, s, t) &= \frac{ustx}{1 - putx}.
\end{cases}
\end{align}

It is easy to see that the functions 
$f_{j,\ell,i}(x, p, q, u, 1, s, t)$ and
$g_{j,\ell,i}(x, p, q, u, 1, s, t)$
can be computed step by step, for $i$ from 0 to $j-2$, in terms of an already computer factor times $F_{j,\ell}(x, p, q,  u, 1, s, t)$ and  adding additional already computerd terms. Substituting these expressions for  $f_{j,\ell,i}(x, p, q, u, 1, s, t)$ and
$g_{j,\ell,i}(x, p, q, u, 1, s, t)$ in
\eqref{eq:5.13},  we can obtain an expression for $F_{j,\ell}(x, p, q, u, 1, s, t).$
Then  we substitute $F_{j,\ell}(x, p, q, u, 1, s, t)$ into \eqref{eq:5.1}--\eqref{eq:5.3} yielding the expression for $F_{j,\ell}$ and completing the proof.
\end{proof}

\begin{rem}
Since
\begin{align*}
\widetilde G_2 &= 1 + uvstx + pu^2vst^2x^2 + \dots,\\
\widetilde G_3 &= 1 + uvstx + (pu^2vst^2 + quv^2s^2t)x^2 \\
&\quad + (pquv^2s^2t^2 + pqu^2vs^2t^2 + pqu^2v^2st^2 + p^2u^3vst^3)x^3 + \dots,\\
\widetilde G_4 &= 1 + uvstx + (pu^2vst^2 + quv^2s^2t)x^2 \\
&\quad +(p^2u^3vst^3+pqu^2v^2st^2+pqu^2vs^2t^2+pqu^2v^2s^2t+pquv^2s^2t^2+q^2uv^3s^3t)x^3+\dots,\\
F_2 &= 1 + uvstx + pu^2vst^2x^2 + \dots, \\
F_3 &= 1 + uvstx + (pu^2vst^2 + quv^2s^2t)x^2 + \dots,\\
F_4 &= 1 + uvstx + (pu^2vst^2 + quv^2s^2t)x^2 + \dots,
\end{align*}
we can obtain $U_n$, $V_n$, $W_n$, $Y_n$ and $Z_n$ when deriving the g.f.\ $F_{j,\ell}$ for any $3 \leq j,\ell \leq 5$.
\end{rem}

\section{Explicit forms of $F_{j,\ell}(x,p,q,u,v,s,t)$ for $3 \leq j,\ell \leq 5$}

In this section, we use Theorem~\ref{thm jll} to derive explicitly the g.f.\ $F_{j,\ell}$ for all $3\le j,\ell\le 5$.

\subsection{The g.f.\ of $S_n(2413,3142,P_3,\widetilde{P}_{3})$ }
\begin{thm}\label{33-thm} The g.f.\ of $S_n(2413,3142,P_3,\widetilde{P}_{3})$ is
\begin{align*}
F_{3,3}(x,p,q,u,v,s,t) = \frac{F_a}{F_b},
\end{align*}
where $F_b = p q u t  x^{2} + p u t x - 1$ and 
\begin{align*}
F_a &= \begin{aligned}[t] &p q \,u^{2} v^{2} s^{2} t^{2} x^{3} - p q \,u^{2} v^{2} s \,t^{2} x^{3} - p q \,u^{2} v \,s^{2} t^{2} x^{3} + p q \,u^{2} v s \,t^{2} x^{3} \\
&+ p q u t \,x^{2} + p u t x - q u \,v^{2} s^{2} t \,x^{2} - u v s t x - 1.
\end{aligned}
\end{align*}
\end{thm}
\begin{proof}
Let $j=\ell=3$ in Theorem \ref{thm jll}.  Since $\widetilde G_2=1+uvstx+\dots$, we have  $W_1=uvst$ and
\begin{align}\label{331}
F_{3,3} &= 1  + f_{3,3,0} + g_{3,3,0}+  f_{3,3,1} + g_{3,3,1}
\end{align}
where
\begin{align}
f_{3,3,0} & = uvstx + puvtx\left(F_{3,3}(x, p, q, u, 1, s, t) - 1\right), \label{332}\\
g_{3,3,0} & = f_{3,3,1} = 0,\label{333} \\
g_{3,3,1} &= quvsx W_1(p,q,1,v,s,t)x\nonumber \\
&\quad + pquvxW_1(p,q,1,v,1,t)x\left( F_{3,3}(x,p,q,u,1,s,t) - 1\right)\nonumber \\
&= quv^2s^2tx^2 + pquv^2tx^2\left( F_{3,3}(x,p,q,u,1,s,t) - 1\right).\label{335}
\end{align}
Substituting \eqref{332}--\eqref{335} into \eqref{331} yields
\begin{align}
F_{3,3} &= 1 + uvstx + quv^2s^2tx^2 + \bigl(puvtx + pquv^2tx^2\bigr)\left( F_{3,3}(x,p,q,u,1,s,t) - 1\right).\label{336}
\end{align}
Letting $v$=1 in \eqref{336}, and solving for  $F_{3,3}(x,p,q,u,1,s,t)$, we have
\begin{align}\label{338}
   F_{3,3}(x,p,q,u,1,s,t)={\frac{p q u t x^{2}+p u t x -q u \,s^{2} t \,x^{2}-u s t x -1}{p q u t x^{2}+p u t x -1}}.
\end{align}
Substituting \eqref{338} into the right-hand side of \eqref{336} yields the desired result.
\end{proof}

\begin{cor}
The number of permutations in $S_{n}(2413,3142,P_3,\widetilde{P}_{3})$ (equivalently, by Proposition~\ref{Proposition1}, in $S_{n}(P_3,\widetilde{P}_{3})$) is given by the g.f.
\begin{align*}
F_{3,3}(x) &= \frac{1}{1 - x - x^2} 
= 1 + x + 2x^2 + 3x^3 + 5x^4 + 8x^5 + 13x^6 + 21x^7 + O(x^8).
\end{align*}
\end{cor}
\begin{proof}
Setting $p=q=u=v=s=t=1$ in Theorem~\ref{33-thm}, we get the desired result.
\end{proof}

\subsection{The g.f.\ of $S_n(2413,3142,P_3,\widetilde{P}_{4})$ }

\begin{thm}\label{34-thm}
The g.f.\ of $S_n(2413,3142,P_3,\widetilde{P}_{4})$ is
\begin{align*}
F_{3,4}(x,p,q,u,v,s,t)= \frac{F_c}{F_d},
\end{align*}
where $F_d = p^{2} q u^{2} t x^{3} + p q u t x^{2} + p u t x - 1$ and 
\begin{align*}
F_c &= \begin{aligned}[t]
&p^{2} q u^{3} v^{2} s^{2} t^{2} x^{4} - p^{2} q u^{3} v s^{2} t^{2} x^{4} - p^{2} q u^{3} v^{2} s t^{2} x^{4} + p^{2} q u^{3} v s t^{2} x^{4} + p q u^{2} v^{2} s^{2} t^{2} x^{3} - p q u^{2} v s^{2} t^{2} x^{3}\\
& \hspace{-1cm} - p q u^{2} v^{2} s^{2} t x^{3} - p q u^{2} v^{2} s t^{2} x^{3} + p q u^{2} v s t^{2} x^{3} + p^{2} q u^{2} t x^{3} - q u v^{2} s^{2} t x^{2} + p q u t x^{2} - u v s t x + p u t x - 1.
\end{aligned}
\end{align*}
\end{thm}

\begin{proof}
Let $j=3$ and $\ell=4$ in Theorem \ref{thm jll}.  Since $G_2=1+uvstx+\dots$, we have  $U_1=uvst$. Since
 $\widetilde G_2=1+uvstx+\dots$, we have  $W_1=uvst$. Therefore,
\begin{align}\label{341}
F_{3,4} &= 1  + f_{3,4,0} + g_{3,4,0} +  f_{3,4,1} + g_{3,4,1}
\end{align}
where
\begin{align}
f_{3,4,0} & = uvstx + puvtx\left(F_{3,4}(x, p, q, u, 1, s, t) - 1\right), \label{342}\\
g_{3,4,0} & = 0,\label{343} \\
f_{3,4,1} &\;= pquvx \left(U_1(p,q,u,1,s,1)xV_1(p,q,1,v,s,t)x\right)\nonumber \\
&\quad + p^2quvx\left(U_1(p,q,u,1,1,1)xV_1(p,q,1,v,1,t)x\right) \left( F_{3,4}(x,p,q,u,1,s,t) - 1\right)\nonumber \\
&= pqu^2v^2s^2tx^3 + p^2qu^2v^2tx^3\left( F_{3,4}(x,p,q,u,1,s,t) - 1\right),\label{344}\\
g_{3,4,1}
&= quvsx W_1(p,q,1,v,s,t)x\nonumber \\
&\quad + pquvxW_1(p,q,1,v,1,t)x\left( F_{3,4}(x,p,q,u,1,s,t) - 1\right)\nonumber \\
&= quv^2s^2tx^2 + pquv^2tx^2\left( F_{3,4}(x,p,q,u,1,s,t) - 1\right),\label{345}
\end{align}
Substituting \eqref{342}--\eqref{345} into \eqref{341} yields
\begin{align}
F_{3,4} &= 1 + uvstx + pqu^2v^2s^2tx^3 + quv^2s^2tx^2 \nonumber\\
&\quad + \bigl(puvtx + p^2qu^2v^2tx^3+ pquv^2tx^2\bigr)\left( F_{3,4}(x,p,q,u,1,s,t) - 1\right).\label{346}
\end{align}
Letting $v$=1 in \eqref{346}, and solving for $F_{3,4}(x,p,q,u,1,s,t)$, we have
\begin{align}\label{348}
   F_{3,4}(x,p,q,u,1,s,t)&=
 \frac{p^2 q u t x^3 - p q u s t x^3 + p q u t x^2 + p u t x - q u s t x^2 - u s t x - 1}{p^2 q u t x^3 + p q u t x^2 + p u t x - 1}.
\end{align}
Substituting \eqref{348} into the right-hand side of \eqref{346} yields the desired result.
\end{proof}
\begin{cor}
The number of permutations in $S_{n}(2413,3142,P_3,\widetilde{P}_{4})$ (equivalently, by Proposition~\ref{Proposition1}, in $S_{n}(P_3,\widetilde{P}_{4})$)  is given by the g.f. 
\begin{align*}
F_{3,4}(x) &= \frac{1}{1-x-x^2-x^3} =1 + x + 2x^2 + 4x^3 + 7x^4 + 13x^5 + 24x^6 + 44x^7 + O(x^8).
\end{align*}
\end{cor}
\begin{proof}
Setting $p=q=u=v=s=t=1$  in Theorem~\ref{34-thm}, we get the desired result.
\end{proof}
\subsection{The g.f.\ of $S_n(2413,3142,P_3,\widetilde{P}_{5})$ }
\begin{thm}\label{35-thm}
The g.f.\ of $S_n(2413,3142,P_3,\widetilde{P}_{5})$ is
\[
F_{3,5}(x,p,q,u,v,s,t)=\frac{F_e}{F_f},
\]
where $F_f = p^{3} q u^{3} t x^{4} + p^{2} q u^{2} t x^{3} + p q u t x^{2} + p u t x - 1$ and
\[
\begin{aligned}
F_e &= \; p^{3} q u^{4} v^{2} s^{2} t^{2} x^{5} - p^{3} q u^{4} v^{2} s t^{2} x^{5} - p^{3} q u^{4} v s^{2} t^{2} x^{5} + p^{3} q u^{4} v s t^{2} x^{5} + p^{3} q u^{3} t x^{4}\\
    &\quad  + p^{2} q u^{3} v^{2} s^{2} t^{2} x^{4} - p^{2} q u^{3} v^{2} s^{2} t x^{4} - p^{2} q u^{3} v^{2} s t^{2} x^{4} - p^{2} q u^{3} v s^{2} t^{2} x^{4} + p^{2} q u^{3} v s t^{2} x^{4} \\
    &\quad + p^{2} q u^{2} t x^{3} + p q u^{2} v^{2} s^{2} t^{2} x^{3} - p q u^{2} v^{2} s^{2} t x^{3} - p q u^{2} v^{2} s t^{2} x^{3} - p q u^{2} v s^{2} t^{2} x^{3}\\
    &\quad  + p q u^{2} v s t^{2} x^{3} + p q u t x^{2} + p u t x - q u v^{2} s^{2} t x^{2} - u v s t x - 1.
\end{aligned}
\]
\end{thm}
\begin{proof}
Let $j=3,\ell=5$ in Theorem \ref{thm jll}. When $i=1$, $U_1=uvst$, $U_2=pu^2vst^2$, $V_1=uvst$, $W_1=uvst$, and
\begin{align}\label{eq:5.17d}
F_{3,5} &= 1  + f_{3,5,0} + g_{3,5,0} + f_{3,5,1} + g_{3,5,1}
\end{align}
where
\begin{align}
f_{3,5,0} & = uvstx + puvtx\left(F_{3,5}(x, p, q, u, 1, s, t) - 1\right), \label{eq:5.18d}\\
g_{3,5,0} & = 0,\label{5.19d}
\end{align}
\begin{align}
f_{3,5,1} &\;= pquvx \Bigl(U_1(p,q,u,1,s,1)xV_{1}(p,q,1,v,s,t)x  \nonumber\\
&\quad + U_2(p,q,u,1,s,1)x^2V_{1}(p,q,1,v,s,t)x \Bigr)\nonumber \\
&\quad + p^2quvx\Bigl(U_1(p,q,u,1,1,1)xV_{1}(p,q,1,v,1,t)x \nonumber \\
&\quad + U_2(p,q,u,1,1,1)x^2V_{1}(p,q,1,v,s,t)x\Bigr)\left(F_{3,5}(x,p,q,u,1,s,t) - 1\right) \nonumber \\
&\;= pqu^2v^2s^2tx^3 + p^2qu^3v^2s^2tx^4 \nonumber \\
&\quad + (p^2qu^2v^2tx^3+p^3qu^3v^2tx^4)\left(F_{3,5}(x,p,q,u,1,s,t) - 1\right),\label{5.20d}\\
g_{3,5,1} 
&= quvsx W_1(p,q,1,v,s,t)x\nonumber \\
&\quad + pquvxW_1(p,q,1,v,1,t)x\left( F_{3,5}(x,p,q,u,1,s,t) - 1\right)\nonumber \\
&= quv^2s^2tx^2 + pquv^2tx^2\left( F_{3,5}(x,p,q,u,1,s,t) - 1\right),\label{5.20e}
\end{align}
Substituting \eqref{eq:5.18d}--\eqref{5.20e} into \eqref{eq:5.17d} yields
\begin{align}
F_{3,5} &= 1 + uvstx + pqu^2v^2s^2tx^3 +p^2qu^3v^2s^2tx^4 + quv^2s^2tx^2 \nonumber\\
&\quad+\big( puvtx+p^2qu^2v^2tx^3+p^3qu^3v^2tx^4+ pquv^2tx^2\big)\left(F_{3,5}(x,p,q,u,1,s,t) - 1\right). \label{45}
\end{align}
Letting $v$=1 in \eqref{45}, and solving for $F_{3,5}(x,p,q,u,1,s,t)$, we have
\begin{equation}\label{eq:3.6}
F_{3,5}(x,p,q,u,1,s,t)= \frac{M_1}{M_2},
\end{equation}
where  $M_2 = p^3 q u^3 t x^4 + p^2 q u^2 t x^3 + p q u t x^2 + p u t x - 1$ and
\[
\begin{aligned}
M_1 = &\; p^3 q u^3 t x^4 - p^2 q u^3 s^2 t x^4 + p^2 q u^2 t x^3 - p q u^2 s^2 t x^3 + p q u t x^2 + p u t x - q u s^2 t x^2 - u s t x - 1.
\end{aligned}
\]
Substituting \eqref{eq:3.6} into the right-hand side of \eqref{45} yields the desired formula for $F_{3,5}$.
\end{proof}

\begin{cor}
The number of permutations in $S_{n}(2413,3142,P_3,\widetilde{P}_{5})$ (equivalently, by Proposition~\ref{Proposition1}, in $S_{n}(P_3,\widetilde{P}_{5})$) is given by the g.f.  
\begin{align*}
F_{3,5}(x) &= \frac{1}{1-x-x^2-x^3-x^4} =  1 + x + 2x^2 + 4x^3 + 8x^4 + 15x^5 + 29x^6 + 56x^7 + O(x^8). 
\end{align*}
\end{cor}

\begin{proof}
Setting $p=q=u=v=s=t=1$ in Theorem~\ref{35-thm}, we get the desired result.
\end{proof}

\subsection{The g.f.\ of $S_n(2413,3142,P_4,\widetilde{P}_{4})$ }
\begin{thm}\label{44-thm}
The g.f.\ of $S_n(2413,3142,P_4,\widetilde{P}_{4})$ is
\begin{align*}
F_{4,4}(x,p,q,u,v,s,t)= \frac{F_g}{F_h},
\end{align*}
where $F_h = p^{3} q u^{2} t^{2} x^{4} + p^{2} q u^{2} t x^{3} + p^{2} q u t^{2} x^{3} + p q^{2} u t x^{3} + p q u t x^{2} + p u t x - 1$ and

\begin{align*}
F_g &= \begin{aligned}[t]
      & p^{3} q^{3} u^{3} v^{3} s^{3} t^{3} x^{7} - p^{3} q^{3} u^{3} v^{3} s^{2} t^{3} x^{7} - p^{3} q^{3} u^{3} v^{2} s^{3} t^{3} x^{7} + p^{3} q^{3} u^{3} v^{2} s^{2} t^{3} x^{7} + p^{2} q^{3} u^{3} v^{3} s^{3} t^{2} x^{6}\\
      & + p^{2} q^{3} u^{2} v^{3} s^{3} t^{3} x^{6} - p^{2} q^{3} u^{3} v^{3} s^{2} t^{2} x^{6} - p^{2} q^{3} u^{3} v^{2} s^{3} t^{2} x^{6} - p^{2} q^{3} u^{2} v^{3} s^{2} t^{3} x^{6} - p^{2} q^{3} u^{2} v^{2} s^{3} t^{3} x^{6}\\
      & + p^{2} q^{3} u^{3} v^{2} s^{2} t^{2} x^{6} + p^{2} q^{3} u^{2} v^{2} s^{2} t^{3} x^{6} + p^{3} q u^{3} v^{2} s^{2} t^{3} x^{5} + p q^{3} u^{2} v^{3} s^{3} t^{2} x^{5} - p^{3} q u^{3} v^{2} s t^{3} x^{5} \\
      & - p^{3} q u^{3} v s^{2} t^{3} x^{5} - p q^{3} u^{2} v^{3} s^{2} t^{2} x^{5} - p q^{3} u^{2} v^{2} s^{3} t^{2} x^{5} + p^{3} q u^{3} v s t^{3} x^{5} + p q^{3} u^{2} v^{2} s^{2} t^{2} x^{5} \\
      &+ p q^{2} u^{2} v^{3} s^{3} t^{2} x^{4} + p^{2} q u^{3} v^{2} s^{2} t^{2} x^{4} + p^{2} q u^{2} v^{2} s^{2} t^{3} x^{4} - p^{2} q u^{3} v^{2} s t^{2} x^{4} - p^{2} q u^{3} v s^{2} t^{2} x^{4}\\
      & - p^{2} q u^{2} v^{2} s^{2} t^{2} x^{4} - p^{2} q u^{2} v^{2} s t^{3} x^{4} - p^{2} q u^{2} v s^{2} t^{3} x^{4} - p q^{2} u^{2} v^{3} s t^{2} x^{4} - p q^{2} u^{2} v s^{3} t^{2} x^{4} \\
      &+ p^{2} q u^{3} v s t^{2} x^{4} + p^{2} q u^{2} v s t^{3} x^{4}+ p q^{2} u^{2} v s t^{2} x^{4} + p q u^{2} v^{2} s^{2} t^{2} x^{3} - q^{2} u v^{3} s^{3} t x^{3} + p^{3} q u^{2} t^{2} x^{4} \\
      &- p q u^{2} v^{2} s^{2} t x^{3} - p q u^{2} v^{2} s t^{2} x^{3} - p q u^{2} v s^{2} t^{2} x^{3} - p q u v^{2} s^{2} t^{2} x^{3} + p q u^{2} v s t^{2} x^{3} + p^{2} q u^{2} t x^{3}\\
      & + p^{2} q u t^{2} x^{3} - q u v^{2} s^{2} t x^{2} + p q^{2} u t x^{3} + p q u t x^{2} - u v s t x + p u t x - 1
\end{aligned}
\end{align*}
\end{thm}

\begin{proof}
Let $j=\ell=4$ in Theorem \ref{thm jll}. When $i=1$, $U_1=uvst$,   $V_1=uvst$, and $W_1=uvst$.  When $i=2$, $U_1=uvst$,  $V_2=pu^2vst^2$, and $W_2=pu^2vst^2+quv^2s^2t$. Then,
\begin{align}\label{eq:5.17}
F_{4,4} &= 1  + f_{4,4,0} + g_{4,4,0} +  f_{4,4,1} + g_{4,4,1} +  f_{4,4,2} + g_{4,4,2} 
\end{align}
where
\begin{align}
f_{4,4,0} & = uvstx + puvtx\left(F_{4,4}(x, p, q, u, 1, s, t) - 1\right), \label{eq:5.18}\\
g_{4,4,0} & = 0,\label{5.19} \\
f_{4,4,1} & = pquvx\Bigl(U_1(p,q,u,1,s,1)xV_{1}(p,q,1,v,s,t)x\Bigr)\nonumber\\
& \quad + p^2quvx\Bigl(U_1(p,q,u,1,1,1)xV_{1}(p,q,1,v,1,t)x\Bigr)\Bigl( F_{4,4}(x,p,q,u,1,s,t) - 1 \Bigr) \nonumber\\
& = pqu^2v^2s^2tx^3 + p^2qu^2v^2tx^3\left(F_{4,4}(x,p,q,u,1,s,t) - 1\right),\label{5.20}\\
g_{4,4,1} 
&= quvsx W_1(p,q,1,v,s,t)x + pquvxW_1(p,q,1,v,1,t)x\left( F_{4,4}(x,p,q,u,1,s,t) - 1\right)\nonumber \\
&= quv^2s^2tx^2 + pquv^2tx^2\left( F_{4,4}(x,p,q,u,1,s,t) - 1\right),\label{5.20a}\\
f_{4,4,2} 
&= pquvx \Bigl(U_1(p,q,u,1,s,1)xV_{2}(p,q,1,v,s,t)x^2 \Bigr)\nonumber\\
&\quad + p^2quvx\Bigl(U_1(p,q,u,1,1,1)xV_{2}(p,q,1,v,1,t)x^2\Bigr)\Bigl( F_{4,4}(x,p,q,u,1,s,t) - 1\Bigr)\nonumber \\
&= p^2qu^2v^2s^2t^2x^4 + p^3qu^2v^2t^2x^4\left( F_{4,4}(x,p,q,u,1,s,t) - 1\right),\label{5.21}\\
g_{4,4,2}
&= quvsx W_2(p,q,1,v,s,t)x^2 + pquvx W_2(p,q,1,v,1,t) x^2 \left(F_{4,4}(x,p,q,u,1,s,t) - 1\right)\nonumber \\
&= q^2uv^3s^3tx^3 + pquv^2t^2s^2x^3 + \bigl(pq^2uv^3tx^3+ p^2quv^2t^2x^3\bigr)\left(F_{4,4}(x,p,q,u,1,s,t) - 1\right).\label{eq:5.22}
\end{align}

Substituting \eqref{eq:5.18}--\eqref{eq:5.22} into \eqref{eq:5.17} yields
\begin{align}\label{eq:5.23}
F_{4,4} &= 1 + uvstx + pqu^2v^2s^2tx^3 + quv^2s^2tx^2 + p^2qu^2v^2s^2t^2x^4+ q^2uv^3s^3tx^3 \nonumber\\
&\quad  + pquv^2t^2s^2x^3  +  \bigl( puvtx  + p^2qu^2v^2tx^3 + pquv^2tx^2 + p^3qu^2v^2t^2x^4 + pq^2uv^3tx^3 \nonumber\\
&\quad + p^2quv^2t^2x^3\bigr)\left(F_{4,4}(x,p,q,u,1,s,t) - 1\right).
\end{align}
Letting $v$=1 in \eqref{eq:5.23}, and solving for $F_{4,4}(x,p,q,u,1,s,t)$, we have
\begin{equation}\label{eq:3.3}
F_{4,4}(x,p,q,u,1,s,t) = \frac{G_1}{G_2},
\end{equation}
where $G_2 = p^{3} q u^{2} t^{2} x^{4} + p^{2} q u^{2} t x^{3} + p^{2} q u t^{2} x^{3} + p q^{2} u t x^{3} + p q u t x^{2} + p u t x - 1$ and
\begin{align*}
G_1 &= \begin{aligned}[t]
      &- p^{2} q u^{2} s^{2} t^{2} x^{4} + p^{3} q u^{2} t^{2} x^{4} - p q u^{2} s^{2} t x^{3} - p q u s^{2} t^{2} x^{3} - q^{2} u s^{3} t x^{3} + p^{2} q u^{2} t x^{3}\\
      & + p^{2} q u t^{2} x^{3} + p q^{2} u t x^{3} - q u s^{2} t x^{2} + p q u t x^{2} + p u t x - u s t x - 1.
\end{aligned}
\end{align*}
Substituting \eqref{eq:3.3} into the right-hand side of \eqref{eq:5.23} yields the desired result.\end{proof}

\begin{cor}
The sets $S_{n}(2413,3142,P_4,\widetilde{P}_{4})$, for $n\geq 0$, are enumerated by the g.f. 
\begin{align*}
F_{4,4}(x) &= \frac{1}{1 - x - x^{2} - 3 x^{3} - x^4} = 1 + x + 2x^2 + 6x^3 + 12x^4 + 25x^5 + 57x^6 + 124x^7 + O(x^8).
\end{align*}
\end{cor}
\begin{proof}
Setting $p=q=u=v=s=t=1$ in Theorem~\ref{44-thm}, we get the desired result.
\end{proof}
\subsection{The g.f.\ of $S_n(2413,3142,P_4,\widetilde{P}_{5})$ }
\begin{thm}\label{45-thm}
The g.f.\ of $S_n(2413,3142,P_4,\widetilde{P}_{5})$ is
\[
F_{4,5}(x,p,q,u,v,s,t)=\frac{F_i}{F_j},
\]
where
$F_j = p^{4} q u^{3} t^{2} x^{5} + p^{3} q u^{3} t x^{4} + p^{3} q u^{2} t^{2} x^{4} + 3 p^{2} q^{2} u^{2} t x^{4} + p^{2} q u^{2} t x^{3} 
     + p^{2} q u t^{2} x^{3} + p q^{2} u t x^{3} + p q u t x^{2} + p u t x - 1$ and 
     
     \begin{small}
     \vspace{-0.5cm}   
\begin{align*}
F_i &= \begin{aligned}[t]
      & p^{5} q^{3} u^{5} v^{3} s^{3} t^{3} x^{9} - p^{5} q^{3} u^{5} v^{3} s^{2} t^{3} x^{9} - p^{5} q^{3} u^{5} v^{2} s^{3} t^{3} x^{9} + p^{5} q^{3} u^{5} v^{2} s^{2} t^{3} x^{9} + p^{4} q^{3} u^{5} v^{3} s^{3} t^{2} x^{8}\\
      & + 2 p^{4} q^{3} u^{4} v^{3} s^{3} t^{3} x^{8} - p^{4} q^{3} u^{5} v^{3} s^{2} t^{2} x^{8} - p^{4} q^{3} u^{5} v^{2} s^{3} t^{2} x^{8} - 2 p^{4} q^{3} u^{4} v^{3} s^{2} t^{3} x^{8} - 2 p^{4} q^{3} u^{4} v^{2} s^{3} t^{3} x^{8}\\
      & - p^{3} q^{4} u^{4} v^{3} s^{3} t^{2} x^{8} + p^{4} q^{3} u^{5} v^{2} s^{2} t^{2} x^{8} + 2 p^{4} q^{3} u^{4} v^{2} s^{2} t^{3} x^{8} + p^{3} q^{4} u^{4} v^{3} s^{2} t^{2} x^{8} + p^{3} q^{4} u^{4} v^{2} s^{3} t^{2} x^{8}\\
      & - p^{3} q^{4} u^{4} v^{2} s^{2} t^{2} x^{8} + 2 p^{3} q^{3} u^{4} v^{3} s^{3} t^{2} x^{7} + 2 p^{3} q^{3} u^{3} v^{3} s^{3} t^{3} x^{7} - 2 p^{3} q^{3} u^{4} v^{3} s^{2} t^{2} x^{7} - 2 p^{3} q^{3} u^{4} v^{2} s^{3} t^{2} x^{7} \\
      &- 2 p^{3} q^{3} u^{3} v^{3} s^{2} t^{3} x^{7} - 2 p^{3} q^{3} u^{3} v^{2} s^{3} t^{3} x^{7} + 2 p^{3} q^{3} u^{4} v^{2} s^{2} t^{2} x^{7} + 2 p^{3} q^{3} u^{3} v^{2} s^{2} t^{3} x^{7} + p^{4} q u^{4} v^{2} s^{2} t^{3} x^{6} \\
      &+ 2 p^{2} q^{3} u^{3} v^{3} s^{3} t^{2} x^{6} + p^{2} q^{3} u^{2} v^{3} s^{3} t^{3} x^{6} - p^{4} q u^{4} v^{2} s t^{3} x^{6} - p^{4} q u^{4} v s^{2} t^{3} x^{6} - 2 p^{2} q^{3} u^{3} v^{3} s^{2} t^{2} x^{6}\\
      & - 2 p^{2} q^{3} u^{3} v^{2} s^{3} t^{2} x^{6} - p^{2} q^{3} u^{2} v^{3} s^{2} t^{3} x^{6} - p^{2} q^{3} u^{2} v^{2} s^{3} t^{3} x^{6} + p^{4} q u^{4} v s t^{3} x^{6} + 2 p^{2} q^{3} u^{3} v^{2} s^{2} t^{2} x^{6} \\
      &+ p^{2} q^{3} u^{2} v^{2} s^{2} t^{3} x^{6} + p^{2} q^{2} u^{3} v^{3} s^{3} t^{2} x^{5} + p^{3} q u^{4} v^{2} s^{2} t^{2} x^{5} + p^{3} q u^{3} v^{2} s^{2} t^{3} x^{5} + p^{2} q^{2} u^{3} v^{3} s^{2} t^{2} x^{5} \\
      &+ p^{2} q^{2} u^{3} v^{2} s^{3} t^{2} x^{5} + p q^{3} u^{2} v^{3} s^{3} t^{2} x^{5} - p^{3} q u^{4} v^{2} s t^{2} x^{5} - p^{3} q u^{4} v s^{2} t^{2} x^{5} - p^{3} q u^{3} v^{2} s^{2} t^{2} x^{5} \\
      &- p^{3} q u^{3} v^{2} s t^{3} x^{5} - p^{3} q u^{3} v s^{2} t^{3} x^{5} - 2 p^{2} q^{2} u^{3} v^{3} s t^{2} x^{5} - 2 p^{2} q^{2} u^{3} v s^{3} t^{2} x^{5} - p q^{3} u^{2} v^{3} s^{2} t^{2} x^{5}\\
      & - p q^{3} u^{2} v^{2} s^{3} t^{2} x^{5} + p^{3} q u^{4} v s t^{2} x^{5} + p^{3} q u^{3} v s t^{3} x^{5} - p^{2} q^{2} u^{3} v^{2} s t^{2} x^{5} - p^{2} q^{2} u^{3} v s^{2} t^{2} x^{5}\\
      & + p q^{3} u^{2} v^{2} s^{2} t^{2} x^{5} + p q^{2} u^{2} v^{3} s^{3} t^{2} x^{4} + 3 p^{2} q^{2} u^{3} v s t^{2} x^{5} + p^{2} q u^{3} v^{2} s^{2} t^{2} x^{4} + p^{2} q u^{2} v^{2} s^{2} t^{3} x^{4} \\
      &- p q^{2} u^{2} v^{3} s^{3} t x^{4} + p^{4} q u^{3} t^{2} x^{5} - p^{2} q u^{3} v^{2} s^{2} t x^{4} - p^{2} q u^{3} v^{2} s t^{2} x^{4} - p^{2} q u^{3} v s^{2} t^{2} x^{4} - p^{2} q u^{2} v^{2} s^{2} t^{2} x^{4}\\
      & - p^{2} q u^{2} v^{2} s t^{3} x^{4} - p^{2} q u^{2} v s^{2} t^{3} x^{4} - p q^{2} u^{2} v^{3} s^{2} t x^{4} - p q^{2} u^{2} v^{3} s t^{2} x^{4} - p q^{2} u^{2} v^{2} s^{3} t x^{4}- p q^{2} u^{2} v s^{3} t^{2} x^{4} \\
      & + p^{2} q u^{3} v s t^{2} x^{4} + p^{2} q u^{2} v s t^{3} x^{4} + p q^{2} u^{2} v s t^{2} x^{4} + p q u^{2} v^{2} s^{2} t^{2} x^{3}- q^{2} u v^{3} s^{3} t x^{3} + p^{3} q u^{3} t x^{4}\\
      &  + p^{3} q u^{2} t^{2} x^{4} - p q u^{2} v^{2} s^{2} t x^{3} - p q u^{2} v^{2} s t^{2} x^{3} - p q u^{2} v s^{2} t^{2} x^{3} - p q u v^{2} s^{2} t^{2} x^{3} + 3 p^{2} q^{2} u^{2} t x^{4} \\
      &+ p q u^{2} v s t^{2} x^{3}+ p^{2} q u^{2} t x^{3} + p^{2} q u t^{2} x^{3} - q u v^{2} s^{2} t x^{2} + p q^{2} u t x^{3} + p q u t x^{2} - u v s t x + p u t x - 1.
\end{aligned}
\end{align*}
\end{small}
\end{thm}

\begin{proof}
Let $j=4, \ell=5$ in Theorem \ref{thm jll}. When $i=1$, $U_1=uvst$,   $U_2=pu^2vst^2+quv^2s^2t$, $V_1=uvst$, and $W_1=uvst$.  When $i=2$, $U_1=uvst$,   $U_2=pu^2vst^2$, and  $W_2=pu^2vst^2+quv^2s^2t$, for the case $a=1$, $V_2=pu^2vst^2+quv^2s^2t$, for the case $a=2$, $V_2=pu^2vst^2$. Then,
\begin{align}\label{451}
F_{4,5} &= 1  + f_{4,5,0} + g_{4,5,0}+  f_{4,5,1} + g_{4,5,1}+  f_{4,5,2} + g_{4,5,2}
\end{align}
where
\begin{align}
f_{4,5,0} & = uvstx + puvtx\left(F_{4,5}(x, p, q, u, 1, s, t) - 1\right), \label{452}\\
g_{4,5,0} & = 0,\label{453} \\
f_{4,5,1}
&= pquvx \Bigl( U_1(p,q,u,1,s,1)  x  V_{1}(p,q,1,v,s,t)x  \nonumber \\
&\quad 
+ U_2(p,q,u,1,s,1)  x^2  V_{1}(p,q,1,v,s,t)  x \Bigr) \nonumber \\
&\quad + p^2quvx \Bigl( U_1(p,q,u,1,1,1)  x V_{1}(p,q,1,v,1,t) x \nonumber \\
&\quad 
 + U_2(p,q,u,1,1,1)  x^2  V_{1}(p,q,1,v,1,t) x \Bigr) \nonumber \\
&\quad \cdot \bigl( F_{4,5}(x,p,q,u,1,s,t) - 1 \bigr) \nonumber \\
&= pqu^2v^2s^2tx^3 + p^2qu^3v^2s^2tx^4 + pq^2u^2v^2s^3tx^4 \nonumber \\
&\quad + \bigl(p^2qu^2v^2tx^3 + p^3qu^3v^2tx^4 + p^2q^2u^2v^2tx^4 \bigr) \bigl(F_{4,5}(x,p,q,u,1,s,t) - 1\bigr) ,\label{454}\\
g_{4,5,1} 
&= quvsx W_1(p,q,1,v,s,t)x \nonumber \\
&\quad + pquvxW_1(p,q,1,v,1,t)x\left( F_{4,5}(x,p,q,u,1,s,t) - 1\right)\nonumber \\
&= quv^2s^2tx^2 + pquv^2tx^2\left( F_{4,5}(x,p,q,u,1,s,t) - 1\right),\label{455}\\
f_{4,5,2}
&= pquvx \Bigl( U_1(p,q,u,1,s,1)  x  V_{2}(p,q,1,v,s,t) x^2 \nonumber \\
&\quad + U_2(p,q,u,1,s,1) x^2  V_{2}(p,q,1,v,s,t) x^2 \Bigr) \nonumber \\
&\quad + p^2quvx \Bigl( U_1(p,q,u,1,1,1)x V_{2}(p,q,1,v,1,t)  x^2 \nonumber \\
&\quad + U_2(p,q,u,1,1,1)x^2  V_{2}(p,q,1,v,1,t)  x^2 \Bigr) \nonumber \\
&\quad \cdot \bigl( F_{4,5}(x,p,q,u,1,s,t) - 1 \bigr) \nonumber \\
&= p^2qu^2v^2s^2t^2x^4 + pq^2u^2v^3s^3tx^4 + p^3qu^3v^2s^2t^2x^5 \nonumber \\
&\quad + (p^3qu^2v^2t^2x^4+p^2q^2u^2v^3tx^4+p^4qu^3v^2t^2x^5)  \bigl( F_{4,5}(x,p,q,u,1,s,t) - 1 \bigr),\label{456}\\
g_{4,5,2}
&= quvsx  W_2(p,q,1,v,s,t)  x^2 \nonumber \\
&\quad  + pquvx W_2(p,q,1,v,1,t)  x^2  \bigl(F_{4,5}(x,p,q,u,1,s,t) - 1\bigr) \nonumber \\
&\quad +pq^2uvxX_1(p,q,1,v,1,1)xZ_1(p,q,u,1,s,1)xY_{1}(p,q,1,v,s,t)x\nonumber \\
           &\quad + p^2q^2uvxX_1(p,q,1,v,1,1)x
               Z_1(p,q,u,1,1,1)x\nonumber \\
&\quad \cdot Y_{1}(p,q,1,v,1,t)x
             \big(F_{4,5}(x,p,q,u,1,s,t)-1\big)\nonumber\\
&= q^2uv^3s^3tx^3 + pquv^2t^2s^2x^3 + pq^2u^2v^3s^2tx^4 \nonumber \\
&\quad + (pq^2uv^3tx^3 +p^2quv^2t^2x^3+p^2q^2u^2v^3tx^4) \bigl(F_{4,5}(x,p,q,u,1,s,t) - 1\bigr).\label{457}
\end{align}

Substituting \eqref{452}--\eqref{457} into \eqref{451} yields
\begin{align}\label{458}
F_{4,5} 
&= 1 + uvstx + puvtx \bigl(F_{4,5}(x, p, q, 1, v, s, t) - 1\bigr) \nonumber \\
&\quad + pqu^2v^2s^2tx^3 + p^2qu^3v^2s^2tx^4 + pq^2u^2v^2s^3tx^4 \nonumber \\
&\quad + (p^2qu^2v^2tx^3+p^3qu^3v^2tx^4+p^2q^2u^2v^2tx^4) \bigl(F_{4,5}(x,p,q,u,1,s,t) - 1\bigr) \nonumber \\
&\quad + quv^2s^2tx^2 + pquv^2tx^2 \bigl(F_{4,5}(x,p,q,u,1,s,t) - 1\bigr) \nonumber \\
&\quad + p^2qu^2v^2s^2t^2x^4 + pq^2u^2v^3s^3tx^4 + p^3qu^3v^2s^2t^2x^5 \nonumber \\
&\quad + (p^3qu^2v^2t^2x^4+p^2q^2u^2v^3tx^4+p^4qu^3v^2t^2x^5) \bigl(F_{4,5}(x,p,q,u,1,s,t) - 1\bigr) \nonumber \\
&\quad + q^2uv^3s^3tx^3 + pquv^2t^2s^2x^3 + pq^2u^2v^3s^2tx^4 \nonumber \\
&\quad + (pq^2uv^3tx^3+p^2quv^2t^2x^3+p^2q^2u^2v^3tx^4) \bigl(F_{4,5}(x,p,q,u,1,s,t) - 1\bigr).
\end{align}
Letting $v$=1 in \eqref{458}, and solving it for $F_{4,5}(x,p,q,u,1,s,t)$, we have 
\begin{equation}\label{4510}
F_{4,5}(x,p,q,u,1,s,t)= \frac{Q_1}{Q_2},
\end{equation}
where
\begin{align*}
Q_2 &= p^{4} q u^{3} t^{2} x^{5} + p^{3} q u^{3} t x^{4} + p^{3} q u^{2} t^{2} x^{4} + 3 p^{2} q^{2} u^{2} t x^{4} + p^{2} q u^{2} t x^{3}+ p^{2} q u t^{2} x^{3} \\
    &\quad  + p q^{2} u t x^{3} + p q u t x^{2} + p u t x - 1, \\[4pt]
Q_1 &= \begin{aligned}[t]
      &- p^{3} q u^{3} s^{2} t^{2} x^{5} + p^{4} q u^{3} t^{2} x^{5} - p^{2} q u^{3} s^{2} t x^{4} - p^{2} q u^{2} s^{2} t^{2} x^{4}- 2 p q^{2} u^{2} s^{3} t x^{4} + p^{3} q u^{3} t x^{4} \\
      & + p^{3} q u^{2} t^{2} x^{4} - p q^{2} u^{2} s^{2} t x^{4} + 3 p^{2} q^{2} u^{2} t x^{4}- p q u^{2} s^{2} t x^{3} - p q u s^{2} t^{2} x^{3} - q^{2} u s^{3} t x^{3} \\
      & + p^{2} q u^{2} t x^{3} + p^{2} q u t^{2} x^{3} + p q^{2} u t x^{3} - q u s^{2} t x^{2} + p q u t x^{2} + p u t x - u s t x - 1.
\end{aligned}
\end{align*}
Substituting \eqref{4510} into the right-hand side of \eqref{458} yields the desired formula for $F_{4,5}$.
\end{proof}

\begin{cor}
The g.f.\ $F_{4,5}(x)$ is given by
\begin{align*}
 &\frac{1}{1-x-x^{2}-3 x^{3}-5 x^{4}-x^5}=1 + x + 2 x^2 + 6 x^3 + 16 x^4 + 34 x^5 + 79 x^6 + 193 x^7 + O(x^8).
\end{align*}
\end{cor}

\begin{proof}
Setting $p=q=u=v=s=t=1$ in Theorem~\ref{45-thm}, we get the desired result.
\end{proof}
\subsection{The g.f.\ of $S_n(2413,3142,P_5,\widetilde{P}_{5})$ }
\begin{thm}\label{55-thm}
The g.f.\ of $S_n(2413,3142,P_5,\widetilde{P}_{5})$ is
\[
F_{5,5}(x,p,q,u,v,s,t)=\frac{F_1+F_2+F_3+F_4+F_5}{F_m},
\]
where
\begin{small}
\begin{align*}
F_m &= p^{5} q u^{3} t^{3} x^{6} + p^{4} q u^{3} t^{2} x^{5} + p^{4} q u^{2} t^{3} x^{5} + 5 p^{3} q^{2} u^{2} t^{2} x^{5} + p^{3} q u^{3} t x^{4} \\
    &\quad + p^{3} q u^{2} t^{2} x^{4} + p^{3} q u t^{3} x^{4} + 3 p^{2} q^{2} u^{2} t x^{4} + 3 p^{2} q^{2} u t^{2} x^{4} + p^{2} q^{2} u t x^{4} \\
    &\quad + p q^{3} u t x^{4} + p^{2} q u^{2} t x^{3} + p^{2} q u t^{2} x^{3} + p q^{2} u t x^{3} + p q u t x^{2} + p u t x - 1,
\end{align*}
\begin{align*}
F_1 &= p^{7} q^{3} u^{5} v^{3} s^{3} t^{5} x^{11} - p^{7} q^{3} u^{5} v^{3} s^{2} t^{5} x^{11} - p^{7} q^{3} u^{5} v^{2} s^{3} t^{5} x^{11} + p^{7} q^{3} u^{5} v^{2} s^{2} t^{5} x^{11} + p^{5} q^{4} u^{4} v^{4} s^{4} t^{4} x^{10} \\
    &\quad + 2 p^{6} q^{3} u^{5} v^{3} s^{3} t^{4} x^{10} + 2 p^{6} q^{3} u^{4} v^{3} s^{3} t^{5} x^{10} - 2 p^{6} q^{3} u^{5} v^{3} s^{2} t^{4} x^{10}- 2 p^{6} q^{3} u^{5} v^{2} s^{3} t^{4} x^{10} + p^{6} q^{3} u^{4} v^{3} s^{3} t^{4} x^{10} \\
    &\quad  - 2 p^{6} q^{3} u^{4} v^{3} s^{2} t^{5} x^{10} - 2 p^{6} q^{3} u^{4} v^{2} s^{3} t^{5} x^{10}- p^{5} q^{4} u^{4} v^{4} s^{2} t^{4} x^{10} - 4 p^{5} q^{4} u^{4} v^{3} s^{3} t^{4} x^{10} - p^{5} q^{4} u^{4} v^{2} s^{4} t^{4} x^{10} \\
    &\quad  + 2 p^{6} q^{3} u^{5} v^{2} s^{2} t^{4} x^{10} - p^{6} q^{3} u^{4} v^{3} s^{2} t^{4} x^{10} - p^{6} q^{3} u^{4} v^{2} s^{3} t^{4} x^{10} + 2 p^{6} q^{3} u^{4} v^{2} s^{2} t^{5} x^{10} + 4 p^{5} q^{4} u^{4} v^{3} s^{2} t^{4} x^{10}\\
    &\quad  + 4 p^{5} q^{4} u^{4} v^{2} s^{3} t^{4} x^{10} + p^{4} q^{4} u^{4} v^{4} s^{4} t^{3} x^{9} + p^{4} q^{4} u^{3} v^{4} s^{4} t^{4} x^{9} + p^{6} q^{3} u^{4} v^{2} s^{2} t^{4} x^{10}- 3 p^{5} q^{4} u^{4} v^{2} s^{2} t^{4} x^{10} \\
    &\quad  + 2 p^{5} q^{3} u^{5} v^{3} s^{3} t^{3} x^{9} + 4 p^{5} q^{3} u^{4} v^{3} s^{3} t^{4} x^{9} + 2 p^{5} q^{3} u^{3} v^{3} s^{3} t^{5} x^{9} + 5 p^{3} q^{5} u^{3} v^{4} s^{4} t^{3} x^{9} - 2 p^{5} q^{3} u^{5} v^{3} s^{2} t^{3} x^{9}\\
    &\quad  - 2 p^{5} q^{3} u^{5} v^{2} s^{3} t^{3} x^{9} + p^{5} q^{3} u^{4} v^{3} s^{3} t^{3} x^{9},
\end{align*}
\begin{align*}
F_2 &= - 4 p^{5} q^{3} u^{4} v^{3} s^{2} t^{4} x^{9} - 4 p^{5} q^{3} u^{4} v^{2} s^{3} t^{4} x^{9} + p^{5} q^{3} u^{3} v^{3} s^{3} t^{4} x^{9}- 2 p^{5} q^{3} u^{3} v^{3} s^{2} t^{5} x^{9} - 2 p^{5} q^{3} u^{3} v^{2} s^{3} t^{5} x^{9}\\
    &\quad   - p^{4} q^{4} u^{4} v^{4} s^{2} t^{3} x^{9} - 4 p^{4} q^{4} u^{4} v^{3} s^{3} t^{3} x^{9} - p^{4} q^{4} u^{4} v^{2} s^{4} t^{3} x^{9}- p^{4} q^{4} u^{3} v^{4} s^{2} t^{4} x^{9} - 4 p^{4} q^{4} u^{3} v^{3} s^{3} t^{4} x^{9}\\
    &\quad   - p^{4} q^{4} u^{3} v^{2} s^{4} t^{4} x^{9} - 3 p^{3} q^{5} u^{3} v^{4} s^{3} t^{3} x^{9} - 3 p^{3} q^{5} u^{3} v^{3} s^{4} t^{3} x^{9} + 2 p^{5} q^{3} u^{5} v^{2} s^{2} t^{3} x^{9} - p^{5} q^{3} u^{4} v^{3} s^{2} t^{3} x^{9} \\
    &\quad - p^{5} q^{3} u^{4} v^{2} s^{3} t^{3} x^{9} + 4 p^{5} q^{3} u^{4} v^{2} s^{2} t^{4} x^{9} - p^{5} q^{3} u^{3} v^{3} s^{2} t^{4} x^{9} - p^{5} q^{3} u^{3} v^{2} s^{3} t^{4} x^{9} + 2 p^{5} q^{3} u^{3} v^{2} s^{2} t^{5} x^{9}\\
    &\quad  + 4 p^{4} q^{4} u^{4} v^{3} s^{2} t^{3} x^{9} + 4 p^{4} q^{4} u^{4} v^{2} s^{3} t^{3} x^{9} + 4 p^{4} q^{4} u^{3} v^{3} s^{2} t^{4} x^{9} + 4 p^{4} q^{4} u^{3} v^{2} s^{3} t^{4} x^{9} - 2 p^{3} q^{5} u^{3} v^{4} s^{2} t^{3} x^{9}\\
    &\quad  + p^{3} q^{5} u^{3} v^{3} s^{3} t^{3} x^{9} - 2 p^{3} q^{5} u^{3} v^{2} s^{4} t^{3} x^{9},
\end{align*}
\begin{align*}
F_3 &= p^{3} q^{4} u^{4} v^{4} s^{4} t^{2} x^{8} + p^{3} q^{4} u^{3} v^{4} s^{4} t^{3} x^{8} + p^{3} q^{4} u^{2} v^{4} s^{4} t^{4} x^{8} + p^{5} q^{3} u^{4} v^{2} s^{2} t^{3} x^{9} + p^{5} q^{3} u^{3} v^{2} s^{2} t^{4} x^{9}\\
    &\quad  - 3 p^{4} q^{4} u^{4} v^{2} s^{2} t^{3} x^{9} - 3 p^{4} q^{4} u^{3} v^{2} s^{2} t^{4} x^{9} + p^{4} q^{3} u^{5} v^{3} s^{3} t^{2} x^{8} + 4 p^{4} q^{3} u^{4} v^{3} s^{3} t^{3} x^{8} + 4 p^{4} q^{3} u^{3} v^{3} s^{3} t^{4} x^{8}\\
    &\quad  + p^{4} q^{3} u^{2} v^{3} s^{3} t^{5} x^{8} + 2 p^{3} q^{5} u^{3} v^{3} s^{2} t^{3} x^{9}+ 2 p^{3} q^{5} u^{3} v^{2} s^{3} t^{3} x^{9} + 3 p^{2} q^{5} u^{3} v^{4} s^{4} t^{2} x^{8} + 3 p^{2} q^{5} u^{2} v^{4} s^{4} t^{3} x^{8}  \\
    &\quad  - p^{4} q^{3} u^{5} v^{3} s^{2} t^{2} x^{8} - p^{4} q^{3} u^{5} v^{2} s^{3} t^{2} x^{8} + p^{4} q^{3} u^{4} v^{3} s^{3} t^{2} x^{8} - 4 p^{4} q^{3} u^{4} v^{3} s^{2} t^{3} x^{8} - 4 p^{4} q^{3} u^{4} v^{2} s^{3} t^{3} x^{8}\\
    &\quad  + p^{4} q^{3} u^{3} v^{3} s^{3} t^{3} x^{8} - 4 p^{4} q^{3} u^{3} v^{3} s^{2} t^{4} x^{8} - 4 p^{4} q^{3} u^{3} v^{2} s^{3} t^{4} x^{8} + p^{4} q^{3} u^{2} v^{3} s^{3} t^{4} x^{8} - p^{4} q^{3} u^{2} v^{3} s^{2} t^{5} x^{8}\\
    &\quad  - p^{4} q^{3} u^{2} v^{2} s^{3} t^{5} x^{8}- p^{3} q^{4} u^{4} v^{4} s^{2} t^{2} x^{8} - p^{3} q^{4} u^{4} v^{3} s^{3} t^{2} x^{8} - p^{3} q^{4} u^{4} v^{2} s^{4} t^{2} x^{8} - p^{3} q^{4} u^{3} v^{4} s^{2} t^{3} x^{8}\\
    &\quad  - 2 p^{3} q^{4} u^{3} v^{3} s^{3} t^{3} x^{8} - p^{3} q^{4} u^{3} v^{2} s^{4} t^{3} x^{8} - p^{3} q^{4} u^{2} v^{4} s^{2} t^{4} x^{8} - p^{3} q^{4} u^{2} v^{3} s^{3} t^{4} x^{8} - p^{3} q^{4} u^{2} v^{2} s^{4} t^{4} x^{8}\\
    &\quad  - 2 p^{2} q^{5} u^{3} v^{4} s^{3} t^{2} x^{8}- 2 p^{2} q^{5} u^{3} v^{3} s^{4} t^{2} x^{8} + p^{2} q^{5} u^{2} v^{4} s^{4} t^{2} x^{8} - 2 p^{2} q^{5} u^{2} v^{4} s^{3} t^{3} x^{8} - 2 p^{2} q^{5} u^{2} v^{3} s^{4} t^{3} x^{8} \\
    &\quad + p^{4} q^{3} u^{5} v^{2} s^{2} t^{2} x^{8}  - p^{4} q^{3} u^{4} v^{3} s^{2} t^{2} x^{8} - p^{4} q^{3} u^{4} v^{2} s^{3} t^{2} x^{8} + 4 p^{4} q^{3} u^{4} v^{2} s^{2} t^{3} x^{8} - p^{4} q^{3} u^{3} v^{3} s^{2} t^{3} x^{8} \\
    &\quad - p^{4} q^{3} u^{3} v^{2} s^{3} t^{3} x^{8} + 4 p^{4} q^{3} u^{3} v^{2} s^{2} t^{4} x^{8} - p^{4} q^{3} u^{2} v^{3} s^{2} t^{4} x^{8} - p^{4} q^{3} u^{2} v^{2} s^{3} t^{4} x^{8} + p^{4} q^{3} u^{2} v^{2} s^{2} t^{5} x^{8} \\
    &\quad + p^{3} q^{4} u^{4} v^{3} s^{2} t^{2} x^{8} + p^{3} q^{4} u^{4} v^{2} s^{3} t^{2} x^{8} + 2 p^{3} q^{4} u^{3} v^{3} s^{2} t^{3} x^{8}, 
\end{align*}
\begin{align*}
F_4 &=  2 p^{3} q^{4} u^{3} v^{2} s^{3} t^{3} x^{8} + p^{3} q^{4} u^{2} v^{3} s^{2} t^{4} x^{8} + p^{3} q^{4} u^{2} v^{2} s^{3} t^{4} x^{8}- p^{2} q^{5} u^{3} v^{4} s^{2} t^{2} x^{8} + p^{2} q^{5} u^{3} v^{3} s^{3} t^{2} x^{8} \\
    &\quad   - p^{2} q^{5} u^{3} v^{2} s^{4} t^{2} x^{8}- p^{2} q^{5} u^{2} v^{4} s^{3} t^{2} x^{8} - p^{2} q^{5} u^{2} v^{4} s^{2} t^{3} x^{8} - p^{2} q^{5} u^{2} v^{3} s^{4} t^{2} x^{8} + p^{2} q^{5} u^{2} v^{3} s^{3} t^{3} x^{8} \\
    &\quad  - p^{2} q^{5} u^{2} v^{2} s^{4} t^{3} x^{8} + p^{2} q^{4} u^{3} v^{4} s^{4} t^{2} x^{7}  + p^{2} q^{4} u^{2} v^{4} s^{4} t^{3} x^{7} + p^{5} q u^{4} v^{2} s^{2} t^{4} x^{7} + p^{4} q^{3} u^{4} v^{2} s^{2} t^{2} x^{8}\\
    &\quad+ p^{4} q^{3} u^{3} v^{2} s^{2} t^{3} x^{8}  + p^{4} q^{3} u^{2} v^{2} s^{2} t^{4} x^{8} - p^{3} q^{4} u^{3} v^{2} s^{2} t^{3} x^{8} +2 p^{3} q^{3} u^{4} v^{3} s^{3} t^{2} x^{7} + 4 p^{3} q^{3} u^{3} v^{3} s^{3} t^{3} x^{7} \\
    &\quad+ 2 p^{3} q^{3} u^{2} v^{3} s^{3} t^{4} x^{7} + p^{2} q^{5} u^{3} v^{3} s^{2} t^{2} x^{8} + p^{2} q^{5} u^{3} v^{2} s^{3} t^{2} x^{8} + p^{2} q^{5} u^{2} v^{3} s^{3} t^{2} x^{8} + p^{2} q^{5} u^{2} v^{3} s^{2} t^{3} x^{8} \\
    &\quad + p^{2} q^{5} u^{2} v^{2} s^{3} t^{3} x^{8} + p q^{5} u^{2} v^{4} s^{4} t^{2} x^{7}- p^{5} q u^{4} v^{2} s t^{4} x^{7} - p^{5} q u^{4} v s^{2} t^{4} x^{7} - 2 p^{3} q^{3} u^{4} v^{3} s^{2} t^{2} x^{7} \\
    &\quad  - 2 p^{3} q^{3} u^{4} v^{2} s^{3} t^{2} x^{7} + p^{3} q^{3} u^{3} v^{3} s^{3} t^{2} x^{7} - 4 p^{3} q^{3} u^{3} v^{3} s^{2} t^{3} x^{7} - 4 p^{3} q^{3} u^{3} v^{2} s^{3} t^{3} x^{7} + p^{3} q^{3} u^{2} v^{3} s^{3} t^{3} x^{7} \\
    &\quad  - 2 p^{3} q^{3} u^{2} v^{3} s^{2} t^{4} x^{7} - 2 p^{3} q^{3} u^{2} v^{2} s^{3} t^{4} x^{7} - p^{2} q^{4} u^{3} v^{4} s^{2} t^{2} x^{7} - p^{2} q^{4} u^{3} v^{2} s^{4} t^{2} x^{7} - p^{2} q^{4} u^{2} v^{4} s^{2} t^{3} x^{7}\\
    &\quad  - p^{2} q^{4} u^{2} v^{2} s^{4} t^{3} x^{7} - p q^{5} u^{2} v^{4} s^{3} t^{2} x^{7} - p q^{5} u^{2} v^{3} s^{4} t^{2} x^{7} + p^{5} q u^{4} v s t^{4} x^{7}+ 2 p^{3} q^{3} u^{4} v^{2} s^{2} t^{2} x^{7}\\
    &\quad - p^{3} q^{3} u^{3} v^{3} s^{2} t^{2} x^{7} - p^{3} q^{3} u^{3} v^{2} s^{3} t^{2} x^{7} + 4 p^{3} q^{3} u^{3} v^{2} s^{2} t^{3} x^{7} - p^{3} q^{3} u^{2} v^{3} s^{2} t^{3} x^{7} - p^{3} q^{3} u^{2} v^{2} s^{3} t^{3} x^{7}\\
     &\quad   + p^{3} q u^{4} v^{2} s^{2} t^{2} x^{5} + p^{3} q u^{3} v^{2} s^{2} t^{3} x^{5} + p^{3} q u^{2} v^{2} s^{2} t^{4} x^{5} + p q u t x^{2} - u v s t x + p u t x\\
\end{align*}
\begin{align*}
F_5 &=  2 p^{3} q^{3} u^{2} v^{2} s^{2} t^{4} x^{7} + p^{3} q^{2} u^{3} v^{3} s^{3} t^{3} x^{6}+ p q^{5} u^{2} v^{3} s^{3} t^{2} x^{7} + p q^{4} u^{2} v^{4} s^{4} t^{2} x^{6} + p^{4} q u^{4} v^{2} s^{2} t^{3} x^{6}\\
    &\quad   + p^{4} q u^{3} v^{2} s^{2} t^{4} x^{6} + p^{3} q^{3} u^{3} v^{2} s^{2} t^{2} x^{7} + p^{3} q^{3} u^{2} v^{2} s^{2} t^{3} x^{7} + 2 p^{3} q^{2} u^{3} v^{3} s^{2} t^{3} x^{6} + 2 p^{3} q^{2} u^{3} v^{2} s^{3} t^{3} x^{6} \\
     &\quad + p^{2} q^{4} u^{3} v^{2} s^{2} t^{2} x^{7} + p^{2} q^{4} u^{2} v^{2} s^{2} t^{3} x^{7} + 2 p^{2} q^{3} u^{3} v^{3} s^{3} t^{2} x^{6} + 2 p^{2} q^{3} u^{2} v^{3} s^{3} t^{3} x^{6}- p^{4} q u^{4} v^{2} s t^{3} x^{6}\\
    &\quad   - p^{4} q u^{4} v s^{2} t^{3} x^{6} - p^{4} q u^{3} v^{2} s^{2} t^{3} x^{6} - p^{4} q u^{3} v^{2} s t^{4} x^{6} - p^{4} q u^{3} v s^{2} t^{4} x^{6}  - 3 p^{3} q^{2} u^{3} v^{3} s t^{3} x^{6}\\
    &\quad  - 3 p^{3} q^{2} u^{3} v s^{3} t^{3} x^{6} - 2 p^{2} q^{3} u^{3} v^{3} s^{2} t^{2} x^{6}- 2 p^{2} q^{3} u^{3} v^{2} s^{3} t^{2} x^{6} + p^{2} q^{3} u^{2} v^{3} s^{3} t^{2} x^{6} - 2 p^{2} q^{3} u^{2} v^{3} s^{2} t^{3} x^{6} \\
    &\quad - 2 p^{2} q^{3} u^{2} v^{2} s^{3} t^{3} x^{6} - p q^{4} u^{2} v^{4} s^{2} t^{2} x^{6} - p q^{4} u^{2} v^{2} s^{4} t^{2} x^{6} + p q^{3} u^{2} v^{4} s^{4} t^{2} x^{5} + p^{4} q u^{4} v s t^{3} x^{6}\\
    &\quad  + p^{4} q u^{3} v s t^{4} x^{6} - 2 p^{3} q^{2} u^{3} v^{2} s t^{3} x^{6} - 2 p^{3} q^{2} u^{3} v s^{2} t^{3} x^{6} + 2 p^{2} q^{3} u^{3} v^{2} s^{2} t^{2} x^{6} - p^{2} q^{3} u^{2} v^{3} s^{2} t^{2} x^{6}\\
    &\quad  - p^{2} q^{3} u^{2} v^{2} s^{3} t^{2} x^{6} + 2 p^{2} q^{3} u^{2} v^{2} s^{2} t^{3} x^{6} + p^{2} q^{2} u^{3} v^{3} s^{3} t^{2} x^{5}+ p^{2} q^{2} u^{2} v^{3} s^{3} t^{3} x^{5} + 5 p^{3} q^{2} u^{3} v s t^{3} x^{6}\\
    &\quad   +p^{2} q^{3} u^{2} v^{2} s^{2} t^{2} x^{6} + p^{2} q^{2} u^{3} v^{3} s^{2} t^{2} x^{5} + p^{2} q^{2} u^{3} v^{2} s^{3} t^{2} x^{5} + p^{2} q^{2} u^{2} v^{3} s^{2} t^{3} x^{5}+ p^{2} q^{2} u^{2} v^{2} s^{3} t^{3} x^{5}\\
    &\quad   + p q^{4} u^{2} v^{2} s^{2} t^{2} x^{6} + p q^{3} u^{2} v^{3} s^{3} t^{2} x^{5} + p^{5} q u^{3} t^{3} x^{6}- p^{3} q u^{4} v^{2} s t^{2} x^{5}  - p^{3} q u^{4} v s^{2} t^{2} x^{5}\\
    &\quad   - p^{3} q u^{3} v^{2} s^{2} t^{2} x^{5} - p^{3} q u^{3} v^{2} s t^{3} x^{5} - p^{3} q u^{3} v s^{2} t^{3} x^{5} - p^{3} q u^{2} v^{2} s^{2} t^{3} x^{5} - p^{3} q u^{2} v^{2} s t^{4} x^{5}\\
    &\quad  - p^{3} q u^{2} v s^{2} t^{4} x^{5}  - 2 p^{2} q^{2} u^{3} v^{3} s t^{2} x^{5} - 2 p^{2} q^{2} u^{3} v s^{3} t^{2} x^{5} - 2 p^{2} q^{2} u^{2} v^{3} s^{2} t^{2} x^{5} - 2 p^{2} q^{2} u^{2} v^{3} s t^{3} x^{5} \\
    &\quad - 2 p^{2} q^{2} u^{2} v^{2} s^{3} t^{2} x^{5}  - 2 p^{2} q^{2} u^{2} v s^{3} t^{3} x^{5} - p q^{3} u^{2} v^{4} s t^{2} x^{5} - p q^{3} u^{2} v^{3} s^{2} t^{2} x^{5} - p q^{3} u^{2} v^{2} s^{3} t^{2} x^{5}\\
    &\quad  - p q^{3} u^{2} v s^{4} t^{2} x^{5} + p^{3} q u^{4} v s t^{2} x^{5} + p^{3} q u^{3} v s t^{3} x^{5}+ p^{3} q u^{2} v s t^{4} x^{5} - p^{2} q^{2} u^{3} v^{2} s t^{2} x^{5}\\
    &\quad   - p^{2} q^{2} u^{3} v s^{2} t^{2} x^{5} - p^{2} q^{2} u^{2} v^{3} s t^{2} x^{5} - p^{2} q^{2} u^{2} v^{2} s t^{3} x^{5} - p^{2} q^{2} u^{2} v s^{3} t^{2} x^{5} - p^{2} q^{2} u^{2} v s^{2} t^{3} x^{5}\\
    &\quad  + p q^{3} u^{2} v^{2} s^{2} t^{2} x^{5}  + p q^{2} u^{2} v^{3} s^{3} t^{2} x^{4} - q^{3} u v^{4} s^{4} t x^{4} + 3 p^{2} q^{2} u^{3} v s t^{2} x^{5}+ 3 p^{2} q^{2} u^{2} v s t^{3} x^{5}  \\
    &\quad + p^{2} q u^{3} v^{2} s^{2} t^{2} x^{4} + p^{2} q u^{2} v^{2} s^{2} t^{3} x^{4} - p q^{2} u^{2} v^{3} s^{3} t x^{4} - p q^{2} u v^{3} s^{3} t^{2} x^{4}+ p^{4} q u^{3} t^{2} x^{5}\\
    &\quad   + p^{4} q u^{2} t^{3} x^{5} + p^{2} q^{2} u^{2} v s t^{2} x^{5} - p^{2} q u^{3} v^{2} s^{2} t x^{4}  - p^{2} q u^{3} v^{2} s t^{2} x^{4}  - p^{2} q u^{3} v s^{2} t^{2} x^{4}\\
    &\quad - p^{2} q u^{2} v^{2} s^{2} t^{2} x^{4} - p^{2} q u^{2} v^{2} s t^{3} x^{4} - p^{2} q u^{2} v s^{2} t^{3} x^{4} - p^{2} q u v^{2} s^{2} t^{3} x^{4} + p q^{3} u^{2} v s t^{2} x^{5}\\
    &\quad - p q^{2} u^{2} v^{3} s^{2} t x^{4}  - p q^{2} u^{2} v^{3} s t^{2} x^{4} - p q^{2} u^{2} v^{2} s^{3} t x^{4} - p q^{2} u^{2} v s^{3} t^{2} x^{4} - p q^{2} u v^{3} s^{3} t x^{4} \\
    &\quad - p q^{2} u v^{3} s^{2} t^{2} x^{4} - p q^{2} u v^{2} s^{3} t^{2} x^{4} + 5 p^{3} q^{2} u^{2} t^{2} x^{5} + p^{2} q u^{3} v s t^{2} x^{4} + p^{2} q u^{2} v s t^{3} x^{4}\\
    &\quad  + p q^{2} u^{2} v s t^{2} x^{4} + p q u^{2} v^{2} s^{2} t^{2} x^{3} - q^{2} u v^{3} s^{3} t x^{3}+ p^{3} q u^{3} t x^{4} + p^{3} q u^{2} t^{2} x^{4} + p^{3} q u t^{3} x^{4}\\
    &\quad   - p q u^{2} v^{2} s^{2} t x^{3} - p q u^{2} v^{2} s t^{2} x^{3} - p q u^{2} v s^{2} t^{2} x^{3} - p q u v^{2} s^{2} t^{2} x^{3} + 3 p^{2} q^{2} u^{2} t x^{4} + 3 p^{2} q^{2} u t^{2} x^{4} \\
    &\quad  + p q u^{2} v s t^{2} x^{3} + p^{2} q^{2} u t x^{4} + p q^{3} u t x^{4} + p^{2} q u^{2} t x^{3} + p^{2} q u t^{2} x^{3} - q u v^{2} s^{2} t x^{2} + p q^{2} u t x^{3}  - 1.
\end{align*}

\end{small}
\end{thm}

\begin{proof}
Let $j=\ell=5$ in Theorem \ref{thm jll}. When $i=1$, $U_1=uvst$,  $U_2=pu^2vst^2+quv^2s^2t$, $V_1=uvst$, and $W_1=uvst$.  When $i=2$, $U_1=uvst$, $U_2=pu^2vst^2+quv^2s^2t$, and $W_2=pu^2vst^2+quv^2s^2t$, for the case $a=1$, $V_2=pu^2vst^2+quv^2s^2t$, for the case $a=2$, $V_2=pu^2vst^2$. When $i=3$, $U_1=uvst$,   $U_2=pu^2vst^2$, and $W_3=p^2u^3vst^3+pqu^2v^2st^2+pqu^2vs^2t^2+pqu^2v^2s^2t+pquv^2s^2t^2+q^2uv^3s^3t$, for the case $a=1$, $V_3=pquv^2s^2t^2+pqu^2vs^2t^2+pqu^2v^2st^2+p^2u^3vst^3$, for the case $a=2$, $V_3=p^2u^3vst^3$.
Then,
\begin{align}\label{551}
F_{5,5} &= 1  + f_{5,5,0} + g_{5,5,0}+  f_{5,5,1} + g_{5,5,1}+  f_{5,5,2} + g_{5,5,2}+  f_{5,5,3} + g_{5,5,3}
\end{align}
where
\begin{align}
f_{5,5,0} & = uvstx + puvtx\left(F_{5,5}(x, p, q, u, 1, s, t) - 1\right), \label{552}\\
g_{5,5,0} & = 0,\label{553} \\ \nonumber
\end{align}
\begin{align}
f_{5,5,1} 
&= pquvx \Bigl( U_1(p,q,u,1,s,1)  x  V_{1}(p,q,1,v,s,t)  x + U_2(p,q,u,1,s,1)  x^2  V_{1}(p,q,1,v,s,t)  x \Bigr) \nonumber \\
&\quad +\ p^2quvx \Bigl( U_1(p,q,u,1,1,1)x V_{1}(p,q,1,v,1,t)  x \nonumber \\
&\quad + U_2(p,q,u,1,1,1) x^2  V_{1}(p,q,1,v,1,t)  x \Bigr)  \bigl( F_{5,5}(x,p,q,u,1,s,t) - 1 \bigr) \nonumber \\
&= pqu^2v^2s^2tx^3 + p^2qu^3v^2s^2tx^4 + pq^2u^2v^2s^3tx^4 \nonumber \\
&\quad + (p^2qu^2v^2tx^3+p^3qu^3v^2tx^4+p^2q^2u^2v^2tx^4 )\bigl(F_{5,5}(x,p,q,u,1,s,t) - 1\bigr) ,\label{554}\\
g_{5,5,1} 
&= quvsxW_1(p,q,1,v,s,t)x + pquvxW_1(p,q,1,v,1,t)x\left( F_{5,5}(x,p,q,u,1,s,t) - 1\right)\nonumber \\
&= quv^2s^2tx^2 + pquv^2tx^2\left( F_{5,5}(x,p,q,u,1,s,t) - 1\right),\label{555}\\ 
f_{5,5,2}
&= pquvx\Bigl( U_1(p,q,u,1,s,1) x  V_{2}(p,q,1,v,s,t)  x^2 + U_2(p,q,u,1,s,1) x^2 V_{2}(x,p,q,1,v,s,t)  x^2 \Bigr) \nonumber \\
&\quad + p^2quvx \Bigl( U_1(p,q,u,1,1,1)xV_{2}(p,q,1,v,1,t) x^2 \nonumber \\
&\quad + U_2(p,q,u,1,1,1)x^2  V_{2}(p,q,1,v,1,t)  x^2 \Bigr)  \bigl( F_{5,5}(x,p,q,u,1,s,t) - 1 \bigr) \nonumber \\
&= p^2qu^2v^2s^2t^2x^4 + pq^2u^2v^3s^3tx^4 + p^3qu^3v^2s^2t^2x^5 \nonumber \\
&\quad +p^2q^2u^2v^2s^3t^2x^5+ p^3qu^2v^2t^2x^4  \bigl( F_{5,5}(x,p,q,u,1,s,t) - 1 \bigr) \nonumber \\
&\quad +(p^2q^2u^2v^3tx^4+p^4qu^3v^2t^2x^5+p^3q^2u^2v^2t^2x^5)\bigl( F_{5,5}(x,p,q,u,1,s,t) - 1 \bigr) ,\label{5561} \\
g_{5,5,2}
&= quvsx  W_2(p,q,1,v,s,t)  x^2+ pquvx W_2(p,q,1,v,1,t)  x^2  \bigl(F_{5,5}(x,p,q,u,1,s,t) - 1\bigr) \nonumber \\
&\quad +pq^2uvxX_1(p,q,1,v,1,1)xZ_1(p,q,u,1,s,1)xY_{1}(p,q,1,v,s,t)x\nonumber \\
           &\quad + p^2q^2uvxX_1(p,q,1,v,1,1)x
               Z_1(p,q,u,1,1,1)x Y_{1}(p,q,1,v,1,t)x \nonumber \\
&\quad \cdot \big(F_{5,5}(x,p,q,u,1,s,t)-1\big)\nonumber\\
&= q^2uv^3s^3tx^3 + pquv^2t^2s^2x^3 + pq^2u^2v^3s^2tx^4 \nonumber \\
&\quad + (pq^2uv^3tx^3+p^2quv^2t^2x^3 +p^2q^2u^2v^3tx^4) \bigl(F_{5,5}(x,p,q,u,1,s,t) - 1\bigr).\label{557}\\
f_{5,5,3}
&= pquvx \Bigl( U_1(p,q,u,1,s,1) x  V_{3}(p,q,1,v,s,t) x^3  + U_2(p,q,u,1,s,1)  x^2  V_{3}(p,q,1,v,s,t)  x^3 \Bigr) \nonumber \\
&\quad + p^2quvx \Bigl( U_1(p,q,u,1,1,1)x V_{3}(p,q,1,v,1,t)x^3 +U_2(p,q,u,1,1,1)x^2  V_{3}(p,q,1,v,1,t) x^3 \Bigr) \nonumber \\
&\quad \cdot \bigl( F_{5,5}(x,p,q,u,1,s,t) - 1 \bigr) \nonumber \\
&= p^2q^2u^2v^3s^3t^2x^5 + p^2q^2u^2v^2s^3t^2x^5 + p^2q^2u^2v^3s^2t^2x^5 + p^3qu^2v^2s^2t^3x^5\nonumber \\
&\quad  + p^4qu^3v^2s^2t^3x^6 + \bigl(p^3q^2u^2v^3t^2x^5 + p^3q^2u^2v^2t^2x^5\nonumber \\
&\quad  + p^3q^2u^2v^3t^2x^5  + p^4qu^2v^2t^3x^5 + p^5qu^3v^2t^3x^6 \bigr)\bigl( F_{5,5}(x,p,q,u,1,s,t) - 1 \bigr),\label{556}\\
g_{5,5,3}
&= quvsx  W_3(p,q,1,v,s,t) x^3 + pquvx W_3(p,q,1,v,1,t)  x^3 \bigl(F_{5,5}(x,p,q,u,1,s,t) - 1\bigr) \nonumber \\
&\quad + pq^2uvx  X_1(p,q,1,v,1,1)  x Z_1(p,q,u,1,s,1)  x Y_{2}(p,q,1,v,s,t) x^2\nonumber \\
&\quad + p^2q^2uvx  X_1(p,q,1,v,1,1)  x Z_1(p,q,u,1,1,1)  x  Y_{2}(p,q,1,v,1,t) x^2 \nonumber \\
&\quad \cdot  \bigl(F_{5,5}(x,p,q,u,1,s,t)-1\bigr) \nonumber \\
&= p^2qu^4v^2s^2t^3x^4 + pq^2u^3v^3s^2t^2x^4 + pq^2u^3v^2s^3t^2x^4 \nonumber \\
&\quad + pq^2u^3v^3s^3tx^4 + pq^2u^2v^3s^3t^2x^4 + q^3u^2v^4s^4tx^4 + p^2q^2u^2v^3s^2t^2x^5\nonumber \\
&\quad  + \bigl(p^3qu^4v^2t^3x^4  + p^2q^2u^3v^3t^2x^4 + p^2q^2u^3v^2t^2x^4  + p^2q^2u^3v^3tx^4 \nonumber \\
&\quad + p^2q^2u^2v^3t^2x^4 + pq^3u^2v^4tx^4  + p^3q^2u^2v^3t^2x^5 \bigr) \bigl(F_{5,5}(x,p,q,u,1,s,t) - 1\bigr),\label{5571}
\end{align}
Substituting \eqref{552}--\eqref{5571} into \eqref{551} yields
\begin{align}\label{558}
F_{5,5}
&= 1 + uvstx + puvtx\left(F_{5,5}(x, p, q, 1, v, s, t) - 1\right)\nonumber \\
&\quad + pqu^2v^2s^2tx^3 + p^2qu^3v^2s^2tx^4 + pq^2u^2v^2s^3tx^4 \nonumber \\
&\quad + (p^2qu^2v^2tx^3+p^3qu^3v^2tx^4+p^2q^2u^2v^2tx^4)\bigl(F_{5,5}(x,p,q,u,1,s,t) - 1\bigr)\nonumber \\
&\quad + quv^2s^2tx^2 + pquv^2tx^2\left( F_{5,5}(x,p,q,u,1,s,t) - 1\right)\nonumber \\
&\quad + p^2qu^2v^2s^2t^2x^4 + pq^2u^2v^3s^3tx^4 + p^3qu^3v^2s^2t^2x^5 \nonumber \\
&\quad + p^2q^2u^2v^2s^3t^2x^5 + p^3qu^2v^2t^2x^4  \bigl( F_{5,5}(x,p,q,u,1,s,t) - 1 \bigr) \nonumber \\
&\quad + (p^2q^2u^2v^3tx^4+p^4qu^3v^2t^2x^5+p^3q^2u^2v^2t^2x^5)\bigl( F_{5,5}(x,p,q,u,1,s,t) - 1 \bigr)\nonumber \\
&\quad + q^2uv^3s^3tx^3 + pquv^2t^2s^2x^3 + pq^2u^2v^3s^2tx^4 \nonumber \\
&\quad + (pq^2uv^3tx^3+p^2quv^2t^2x^3 +p^2q^2u^2v^3tx^4) \bigl(F_{5,5}(x,p,q,u,1,s,t) - 1\bigr)\nonumber \\
&\quad + p^2q^2u^2v^3s^3t^2x^5 + p^2q^2u^2v^2s^3t^2x^5 + p^2q^2u^2v^3s^2t^2x^5 + p^3qu^2v^2s^2t^3x^5\nonumber \\
&\quad + p^4qu^3v^2s^2t^3x^6 + \bigl(p^3q^2u^2v^3t^2x^5 + p^3q^2u^2v^2t^2x^5\nonumber \\
&\quad + p^3q^2u^2v^3t^2x^5  + p^4qu^2v^2t^3x^5 + p^5qu^3v^2t^3x^6 \bigr)\bigl( F_{5,5}(x,p,q,u,1,s,t) - 1 \bigr)\nonumber \\
&\quad + p^2qu^4v^2s^2t^3x^4 + pq^2u^3v^3s^2t^2x^4 + pq^2u^3v^2s^3t^2x^4 \nonumber \\
&\quad + pq^2u^3v^3s^3tx^4 + pq^2u^2v^3s^3t^2x^4 + q^3u^2v^4s^4tx^4 + p^2q^2u^2v^3s^2t^2x^5\nonumber \\
&\quad + \bigl(p^3qu^4v^2t^3x^4  + p^2q^2u^3v^3t^2x^4 + p^2q^2u^3v^2t^2x^4  + p^2q^2u^3v^3tx^4 \nonumber \\
&\quad + p^2q^2u^2v^3t^2x^4 + pq^3u^2v^4tx^4  + p^3q^2u^2v^3t^2x^5 \bigr) \bigl(F_{5,5}(x,p,q,u,1,s,t) - 1\bigr)
\end{align}
Letting $v$=1 in \eqref{558}, and solving for $F_{5,5}(x,p,q,u,1,s,t)$, we have
\begin{equation}\label{5510}
F_{5,5}(x,p,q,u,1,s,t)= \frac{E_1}{E_2},
\end{equation}
where
\begin{align*}
E_2 &= p^{5} q u^{3} t^{3} x^{6} + p^{4} q u^{3} t^{2} x^{5} + p^{4} q u^{2} t^{3} x^{5} + 5 p^{3} q^{2} u^{2} t^{2} x^{5} + p^{3} q u^{3} t x^{4} \\
    &\quad + p^{3} q u^{2} t^{2} x^{4} + p^{3} q u t^{3} x^{4} + 3 p^{2} q^{2} u^{2} t x^{4} + 3 p^{2} q^{2} u t^{2} x^{4} + p^{2} q^{2} u t x^{4} \\
    &\quad + p q^{3} u t x^{4} + p^{2} q u^{2} t x^{3} + p^{2} q u t^{2} x^{3} + p q^{2} u t x^{3} + p q u t x^{2} + p u t x - 1, \\[4pt]
E_1 &= \begin{aligned}[t]
      &- p^{4} q u^{3} s^{2} t^{3} x^{6} + p^{5} q u^{3} t^{3} x^{6} - p^{3} q u^{3} s^{2} t^{2} x^{5} - p^{3} q u^{2} s^{2} t^{3} x^{5} - 3 p^{2} q^{2} u^{2} s^{3} t^{2} x^{5}\\
      & + p^{4} q u^{3} t^{2} x^{5} + p^{4} q u^{2} t^{3} x^{5} - 2 p^{2} q^{2} u^{2} s^{2} t^{2} x^{5} + 5 p^{3} q^{2} u^{2} t^{2} x^{5} - p^{2} q u^{3} s^{2} t x^{4}\\
      & - p^{2} q u^{2} s^{2} t^{2} x^{4} - p^{2} q u s^{2} t^{3} x^{4} - 2 p q^{2} u^{2} s^{3} t x^{4} - 2 p q^{2} u s^{3} t^{2} x^{4} - q^{3} u s^{4} t x^{4} \\
      &+ p^{3} q u^{3} t x^{4} + p^{3} q u^{2} t^{2} x^{4} + p^{3} q u t^{3} x^{4} - p q^{2} u^{2} s^{2} t x^{4} - p q^{2} u s^{3} t x^{4} - p q^{2} u s^{2} t^{2} x^{4}  \\
      & + 3 p^{2} q^{2} u^{2} t x^{4}+ 3 p^{2} q^{2} u t^{2} x^{4} + p^{2} q^{2} u t x^{4} + p q^{3} u t x^{4} - p q u^{2} s^{2} t x^{3} - p q u s^{2} t^{2} x^{3} \\
      &  q^{2} u s^{3} t x^{3}+ p^{2} q u^{2} t x^{3} + p^{2} q u t^{2} x^{3} + p q^{2} u t x^{3} - q u s^{2} t x^{2} + p q u t x^{2} + p u t x - u s t x - 1.
\end{aligned}
\end{align*}
Substituting \eqref{5510} into the right-hand side of \eqref{558} yields the desired formula for $F_{5,5}$.
\end{proof}

\begin{cor}
The g.f.\ $F_{4,5}(x)$ is given by  
\begin{align*}
\frac{1}{1-x-x^2-3x^3-11x^4-7x^5-x^6}=1 + x + 2x^2 + 6x^3 + 22x^4 + 58x^5 + 137x^6 + 385x^7 + O(x^8).
\end{align*}
\end{cor}
\begin{proof}
Setting $p=q=u=v=s=t=1$  in Theorem~\ref{55-thm}, we get the desired result.
\end{proof}

\section{Concluding remarks}

In this paper, we initiated the study of permutations simultaneously avoiding more than one partially ordered pattern, thereby extending the results of \cite{GaoKitaev2019}, \cite{Gao2026}, and several other related works. While we obtained several explicit enumerative results, we were not able to generalize the most comprehensive result of \cite{Gao2026}, namely the system of equations that yields generating functions for separable permutations avoiding an arbitrary POP $P_j$.

Indeed, imposing the additional avoidance of $\widetilde{P}_{\ell}$ does not significantly simplify the enumeration problem. Although avoiding more patterns often leads to a more rigid structural description of the permutations involved, in our setting the number of cases to analyze increases substantially, and the resulting generating functions become considerably more intricate. This increased complexity is reflected, for example, in the large number of monomials appearing in both the numerator and the denominator of the rational generating functions we derive.

An interesting open problem is whether one can develop a general system of equations that determines the number of separable permutations simultaneously avoiding $P_j$ and $\widetilde{P}_{\ell}$ for $j,\ell \geq 4$ (if one of $j$ or $\ell$ equals $3$, the restriction to separable permutations is superfluous by Propostion~\ref{Proposition1}). If such a system exists, it would be natural to investigate the computational complexity of deriving the corresponding generating function for arbitrary values of $j$ and $\ell$.

\section*{Acknowledgements.}
The work of the first, second, and fourth authors is supported by the Fundamental Research Funds for the Central Universities.

\end{document}